\documentclass[12pt]{article}
\usepackage{amsfonts,amssymb}
\usepackage{url}
\usepackage{caption,subcaption}
\usepackage{amsmath,amsfonts,graphicx,tikz,pgf,endnotes,multirow,blkarray,amsthm,bigdelim,arydshln,scalefnt,verbatim} 
\usetikzlibrary{decorations.text,shadings,shadows,shadows.blur,shapes.symbols,3d,calc}
\usetikzlibrary[patterns]
  \usetikzlibrary{shapes.geometric}
\usetikzlibrary{calc,3d,intersections, positioning,intersections,shapes}
\usepackage{xcolor}
\usepackage{tikz-3dplot}
\usepackage{pgfplots}

\newtheorem{defin}{Definition}[section]

\newtheorem{proposition}[defin]{Proposition}

\newtheorem{lemma}[defin]{Lemma}

\newtheorem{thm}[defin]{Theorem}
\newtheorem{theorem}[defin]{Theorem}

\newtheorem{corollary}[defin]{Corollary}

\newtheorem{claim}{Claim}

\def\qed{\hbox{\kern1pt\vrule height6pt width4pt
depth1pt\kern1pt}\medskip}

\newcommand{\R}{{\mathbb{R}}}

\newcommand{\Q}{{\mathbb{Q}}}
\newcommand{\A}{{\mathbb{A}}}

\newcommand{\scra}{{\mathcal{A}}}

\newcommand{\scrr}{{\mathcal{R}}}

\newcommand{\Lo}{{\rm Loop}}

\newcommand{\sm}{\setminus}

\newcommand{\rank}{\mbox{\rm rank }}

\definecolor{blue-green}{rgb}{0.0, 0.87, 0.87}

\title{Point-hyperplane frameworks, slider joints, and rigidity preserving transformations}
\date{}
\author{Yaser Eftekhari\footnote{York University, Toronto, Canada, yas.eftekhari@gmail.com, supported in part by NSERC, Canada}, Bill Jackson\footnote{Queen Mary, University of London, UK, b.jackson@qmul.ac.uk}, Anthony Nixon\footnote{Lancaster University, UK, a.nixon@lancaster.ac.uk}, Bernd Schulze\footnote{Lancaster University, UK, b.schulze@lancaster.ac.uk, supported by EPSRC Grant EP/M013642/1},
\\ Shin-ichi Tanigawa\footnote{Kyoto University, Kyoto, Japan, and Centrum Wiskunde \& Informatica (CWI), Amsterdam, The Netherlands, supported by JSPS Postdoctoral Fellowships for Research Abroad and JSPS Grant-in-Aid for Scientific Research(A)(25240004)} and Walter Whiteley\footnote{York University, Toronto, Canada, whiteley@mathstat.yorku.ca, supported in part by NSERC, Canada}}
\begin{document}
\maketitle

\begin{abstract}
A one-to-one correspondence between the infinitesimal motions of bar-joint frameworks in $\mathbb{R}^d$ and those in $\mathbb{S}^d$ is a classical observation by Pogorelov, and further connections among different rigidity models in various different spaces have been extensively studied. In this paper, we shall extend this line of research to include the infinitesimal rigidity of frameworks consisting of points and hyperplanes. This enables us to understand correspondences between point-hyperplane rigidity, classical bar-joint rigidity, and scene analysis.

Among other results, we derive a combinatorial characterization of graphs that can be realized as infinitesimally rigid frameworks in the plane with a given set of points collinear. This extends a result by Jackson and Jord\'{a}n, which deals with the case when three points are collinear.
\end{abstract}


\section{Introduction}
Given a collection of objects in a space satisfying particular geometric constraints, a fundamental question is whether the given constraints uniquely determine the whole configuration up to congruence. The rigidity problem for bar-joint frameworks in $\mathbb{R}^d$, where the objects are points, the constraints are pairwise distances and only local deformations are considered, is a classical example.

Pogorelov~\cite[Chapter V]{Pog} observed that the space of infinitesimal motions of a bar-joint framework on a semi-sphere is isomorphic to those of the framework obtained by a central projection to Euclidean space. Since then, connections between various types of rigidity models in different spaces have been extensively studied, see, e.g.,~\cite{C&W,CrWh,Izmestiev,SaliolaWh,BSWWSphere,polarity1,polarity2}. When talking about  infinitesimal rigidity, these connections are often just consequences of the fact that  infinitesimal rigidity is preserved by projective transformations. A key essence of the research is its geometric and combinatorial interpretations, which sometimes give us unexpected connections between theory and real applications.

In this paper we shall extend this line of research to include {\em point-hyperplane rigidity}. A {\em point-hyperplane framework} consists of points and hyperplanes along with point-point distance constraints, point-hyperplane distance constraints, and hyperplane-hyperplane angle constraints.  The $2$-dimensional point-line version has been considered, for example in \cite{JO,Owen,Whlear}, for a possible application to CAD.
We will show  that the infinitesimal rigidity of a point-hyperplane framework is closely related to that of a bar-joint framework with {\em nongeneric} positions for its joints.
Understanding the infinitesimal rigidity of such nongeneric bar-joint frameworks  is a classical but still  challenging problem,
and our results  give new insight into this problem.

Specifically, in Section 2 we establish a one-to-one correspondence between the space of infinitesimal motions of a point-hyperplane framework and that of a bar-joint framework with a given set of joints in the same hyperplane by extending the correspondence between Euclidean rigidity and spherical rigidity. Combining this with a result by Jackson and Owen~\cite{JO} for point-line rigidity, we give a combinatorial characterization of a graph that can be realized as an infinitesimally rigid bar-joint framework in the plane with a given set of points collinear. This extends a result by Jackson and Jord\'{a}n~\cite{JJcol}, which deals with the case when three points are collinear.

Let us denote the underlying graph of a point-hyperplane framework in $\mathbb{R}^d$ by $G=(V_P\cup V_L, E_{PP}\cup E_{PL}\cup E_{LL})$, where $V_P$ and $V_L$ represent the set of points and the set of hyperplanes, respectively. The edge set is partitioned into $E_{PP}, E_{LP}, E_{LL}$ according to the bipartition $\{V_P, V_L\}$ of the vertex set.
Each $i\in V_P$ is associated with $p_i\in \mathbb{R}^d$ while each $j\in V_L$ is associated with a hyperplane $\{x\in \mathbb{R}^d: \langle a_j, x\rangle+r_j=0\}$ for some $a_j\in \mathbb{S}^{d-1}$ and $r_j\in \mathbb{R}$.
We will see in Section 2.2 that the  infinitesimal motions of the framework are given by the  solutions of the following system of linear equations in $\dot{p}_i, \dot{a}_j, \dot{r}_j$:
\begin{align}
\langle p_i - p_j, \dot{p}_i-\dot{p}_j\rangle&=0 && (ij\in E_{PP}) \nonumber \\
\langle p_i, \dot{a}_j\rangle+\langle \dot{p}_i, a_j\rangle +\dot{r}_j&=0 && (ij\in E_{PL}) \nonumber\\
\langle a_i, \dot{a}_j\rangle+\langle \dot{a}_i, a_j\rangle&=0 && (ij\in E_{LL}) \nonumber\\
\langle a_i, \dot{a}_i\rangle&=0 && (i\in V_L).\nonumber
\end{align}

Now observe that, if $V_L=\emptyset$ then the system is exactly that of a bar-joint framework on $V_P$ in Euclidean space while, if $V_P=\emptyset$ then the system becomes that of a bar-joint framework on $V_L$ in spherical space. Hence the system of point-hyperplane frameworks is a mixture of
these two settings.
Further detailed restrictions of the system enable us to link various types of rigidity models with point-hyperplane rigidity. In  the second part of the paper, the following new results are obtained:
\begin{itemize}
\item If $\dot{a}_j=0\ (j\in V_L)$, the system models the case when the normal of each hyperplane is fixed. Such a rigidity model was investigated by Owen and Power~\cite{OP} for $d=2$. We show how to derive their combinatorial characterization in the plane from the result of Jackson and Owen~\cite{JO}.

When $E_{PP}=E_{LL}=\emptyset$, we further point out a connection to the parallel drawing problem from scene analysis, and we derive a combinatorial characterization of graphs $G=(V_P\cup V_L, E_{PL})$ which can be realized as a fixed-normal rigid point-hyperplane framework in $\mathbb{R}^d$ using a theorem of Whiteley~\cite{WhHypergraph}.

\item If $\dot{r}_j=0\ (j\in V_L)$, the system can model the case when concurrent hyperplanes can rotate around a common intersection point. We derive a characterization of graphs which can be realized as a rigid point-line framework in the plane in this rigidity model.
By using the rigidity transformation established in Section 2, this result is translated to a characterization of the infinitesimal rigidity of bar-joint frameworks in the plane with horizontal slider-joints on a line. Our result allows us to put slider points anywhere on this line.

\item If $\dot{a}_j=\dot{r}_j=0\ (j\in V_L)$, the system models the case when each hyperplane is fixed. A combinatorial characterization is derived for $d=2$ by first transforming the point-line framework to a bar-joint framework (with nongeneric positions for its joints) and then applying a theorem by Servatius et al~\cite{ShaiServWh}.
\end{itemize}

Point-line frameworks in the plane with different types of constraints imposed on the lines may be used to model structures in engineering with various types of slider-joints (e.g. linkages with prismatic joints in mechanical engineering). 
Indeed, the use of slider-joints in both mechanical and civil engineering provides a key motivation for our work.
The following example from \cite{RW} illustrates how our results may be applied to problems involving slider-joints in engineering, see also \cite{KatTan, Mys, StreinuTheran}.
Consider the `sliding pair chain' shown in Figure~\ref{fig:DCa}(a) consisting of four rigid bodies (labelled $B_1$, $B_2$, $B_3$, $B_4$) connected at five slider joints (labeled $\ell_1,\ell_2,\ldots,\ell_5$).
Each slider joint constrains the relative motion between its two incident bodies to be a translation in a direction determined by the orientation of the slider joint.
We may model this system as a  point-line framework, with each body  represented by a `bar' i.e. two points joined by a distance constraint, as indicated in Figure~\ref{fig:DCa}(b). We will see in Section~\ref{subsec:jj} that this framework has one degree of freedom.

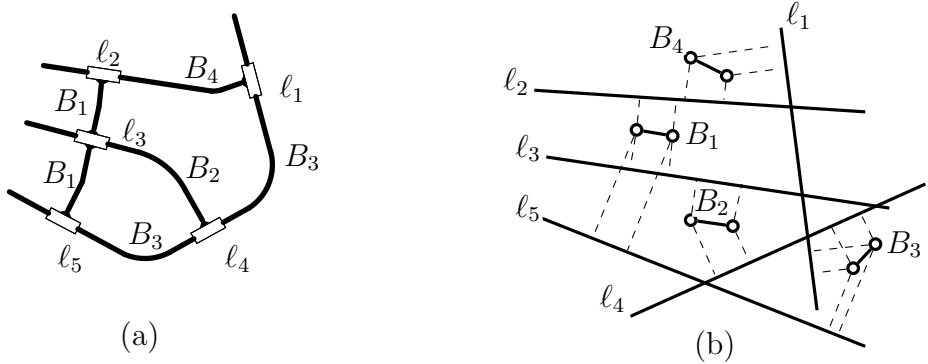
\begin{figure}[htp]
\begin{center}
\begin{tikzpicture}[line cap=round,line join=round,x=1.0cm,y=1.0cm,scale=0.76]
\draw(4.096441487877306,1.941029383754102) -- (4.064258660214266,1.7368450717971724) -- (4.64055244353629,1.6512568861552879) -- (4.669893671268323,1.8546187807773733) -- cycle;
\draw(3.910199926058874,0.8353161830359884) -- (3.853141135630951,0.6299045374954655) -- (4.412317281824597,0.48155168238286555) -- (4.475081951295313,0.6869633279233885) -- cycle;
\draw(3.4522103510220505,-0.5360730814609189) -- (3.3558738816018825,-0.7223235890065769) -- (3.869668385176112,-0.9920657033830471) -- (3.96600485459628,-0.805815195837389) -- cycle;
\draw(6.983131896301509,2.0014894476809237) -- (6.778638063679632,1.9471634729826215) -- (6.929026548715345,1.3825554049068356) -- (7.130916756357505,1.4404266024344958) -- cycle;
\draw(6.397291156476958,-0.7035909688865847) -- (6.501486006256852,-0.8830376546186219) -- (5.9978775656540355,-1.1608905873650024) -- (5.893682715874144,-0.9814439016329648) -- cycle;
\fill[fill=black,fill opacity=1.0] (3.7412197592825547,-0.5874525318183418) -- (3.602995105438134,-0.6152350775293627) -- (3.797535414384365,-0.7173687397261338) -- cycle;
\fill[fill=black,fill opacity=1.0] (4.109905692556604,0.4644340452544886) -- (4.044606349239949,0.5791076440889966) -- (4.224110851387556,0.5314840006620807) -- cycle;
\fill[fill=black,fill opacity=1.0] (4.201199757241282,0.846727941121573) -- (4.089125939485931,0.7883255128430241) -- (4.269724677042619,0.7408955413634898) -- cycle;
\fill[fill=black,fill opacity=1.0] (4.326729096182713,1.6113157328557417) -- (4.23041241155347,1.7121687720933303) -- (4.449615103506464,1.6796139168527866) -- cycle;
\fill[fill=black,fill opacity=1.0] (6.764002582802883,1.6447497983062986) -- (6.824233096865494,1.7759846517131102) -- (6.879251804492295,1.5694262425781416) -- cycle;
\fill[fill=black,fill opacity=1.0] (6.096283812668379,-0.7556883937765311) -- (6.05834808703073,-0.8905940416845036) -- (6.228334966136731,-0.7968081773501582) -- cycle;
\draw [shift={(6.420652588310527,0.21422257545746157)},line width=2.pt]  plot[domain=-1.0019484683735378:0.25168997175818697,variable=\t]({1.*0.8991146585813408*cos(\t r)+0.*0.8991146585813408*sin(\t r)},{0.*0.8991146585813408*cos(\t r)+1.*0.8991146585813408*sin(\t r)});
\draw [shift={(5.091880956963935,-0.5060461592911578)},line width=2.pt]  plot[domain=4.309273391165991:5.206330349304271,variable=\t]({1.*0.9180401029873122*cos(\t r)+0.*0.9180401029873122*sin(\t r)},{0.*0.9180401029873122*cos(\t r)+1.*0.9180401029873122*sin(\t r)});
\draw [shift={(4.607454448474973,-0.8533597885840692)},line width=2.pt]  plot[domain=0.47987248375307245:1.2897984949530479,variable=\t]({1.*1.3440862514658878*cos(\t r)+0.*1.3440862514658878*sin(\t r)},{0.*1.3440862514658878*cos(\t r)+1.*1.3440862514658878*sin(\t r)});
\draw [line width=2.pt] (5.799731265251372,-0.23284077714512977)-- (6.096283812668379,-0.7556883937765311);
\draw (4.096441487877306,1.941029383754102)-- (4.064258660214266,1.7368450717971724);
\draw (4.064258660214266,1.7368450717971724)-- (4.64055244353629,1.6512568861552879);
\draw (4.64055244353629,1.6512568861552879)-- (4.669893671268323,1.8546187807773733);
\draw (4.669893671268323,1.8546187807773733)-- (4.096441487877306,1.941029383754102);
\draw (3.910199926058874,0.8353161830359884)-- (3.853141135630951,0.6299045374954655);
\draw (3.853141135630951,0.6299045374954655)-- (4.412317281824597,0.48155168238286555);
\draw (4.412317281824597,0.48155168238286555)-- (4.475081951295313,0.6869633279233885);
\draw (4.475081951295313,0.6869633279233885)-- (3.910199926058874,0.8353161830359884);
\draw (3.4522103510220505,-0.5360730814609189)-- (3.3558738816018825,-0.7223235890065769);
\draw (3.3558738816018825,-0.7223235890065769)-- (3.869668385176112,-0.9920657033830471);
\draw (3.869668385176112,-0.9920657033830471)-- (3.96600485459628,-0.805815195837389);
\draw (3.96600485459628,-0.805815195837389)-- (3.4522103510220505,-0.5360730814609189);
\draw (6.983131896301509,2.0014894476809237)-- (6.778638063679632,1.9471634729826215);
\draw (6.778638063679632,1.9471634729826215)-- (6.929026548715345,1.3825554049068356);
\draw (6.929026548715345,1.3825554049068356)-- (7.130916756357505,1.4404266024344958);
\draw (7.130916756357505,1.4404266024344958)-- (6.983131896301509,2.0014894476809237);
\draw (6.397291156476958,-0.7035909688865847)-- (6.501486006256852,-0.8830376546186219);
\draw (6.501486006256852,-0.8830376546186219)-- (5.9978775656540355,-1.1608905873650024);
\draw (5.9978775656540355,-1.1608905873650024)-- (5.893682715874144,-0.9814439016329648);
\draw (5.893682715874144,-0.9814439016329648)-- (6.397291156476958,-0.7035909688865847);
\draw [line width=2.pt] (6.255733220946309,1.50631402422271)-- (4.657436758996908,1.7490679430451415);
\draw [line width=2.pt] (4.080085439672488,1.8367577419278602)-- (3.3160459730147536,1.9528022800230946);
\draw [line width=2.pt] (3.0229838870597003,0.9511761285833096)-- (3.8788897481252818,0.726643002524263);
\draw [line width=2.pt] (4.980115118813288,0.43775425175875726)-- (4.442411841064472,0.5788120518476931);
\draw [line width=2.pt] (4.731746589589626,-1.3504991586516082)-- (3.918134317814373,-0.9085369369465321);
\draw [line width=2.pt] (3.4055703330288614,-0.6301071180506987)-- (2.719961502877965,-0.25767763006749594);
\draw [line width=2.pt] (5.526525883105344,-1.313243879268059)-- (5.94647968887873,-1.0786750868540955);
\draw [line width=2.pt] (6.460308871834425,-0.7916713990770411)-- (6.904971220296668,-0.5433014386747071);
\draw [line width=2.pt] (7.289942440593343,0.43775425175875726)-- (7.029738409150239,1.4123366240729276);
\draw [line width=2.pt] (4.326729096182713,1.6113157328557417)-- (4.286787942883167,1.2347277160314496);
\draw [line width=2.pt] (4.201199757241282,0.846727941121573)-- (4.286787942883167,1.2347277160314496);
\draw [line width=2.pt] (4.109905692556604,0.4644340452544886)-- (3.9948998146598345,-0.08234732725278203);
\draw [line width=2.pt] (3.9948998146598345,-0.08234732725278203)-- (3.7412197592825547,-0.5874525318183418);
\draw (3.7412197592825547,-0.5874525318183418)-- (3.602995105438134,-0.6152350775293627);
\draw (3.602995105438134,-0.6152350775293627)-- (3.797535414384365,-0.7173687397261338);
\draw (3.797535414384365,-0.7173687397261338)-- (3.7412197592825547,-0.5874525318183418);
\draw (4.109905692556604,0.4644340452544886)-- (4.044606349239949,0.5791076440889966);
\draw (4.044606349239949,0.5791076440889966)-- (4.224110851387556,0.5314840006620807);
\draw (4.224110851387556,0.5314840006620807)-- (4.109905692556604,0.4644340452544886);
\draw (4.201199757241282,0.846727941121573)-- (4.089125939485931,0.7883255128430241);
\draw (4.089125939485931,0.7883255128430241)-- (4.269724677042619,0.7408955413634898);
\draw (4.269724677042619,0.7408955413634898)-- (4.201199757241282,0.846727941121573);
\draw (4.326729096182713,1.6113157328557417)-- (4.23041241155347,1.7121687720933303);
\draw (4.23041241155347,1.7121687720933303)-- (4.449615103506464,1.6796139168527866);
\draw (4.449615103506464,1.6796139168527866)-- (4.326729096182713,1.6113157328557417);
\draw (6.764002582802883,1.6447497983062986)-- (6.824233096865494,1.7759846517131102);
\draw (6.824233096865494,1.7759846517131102)-- (6.879251804492295,1.5694262425781416);
\draw (6.879251804492295,1.5694262425781416)-- (6.764002582802883,1.6447497983062986);
\draw (6.096283812668379,-0.7556883937765311)-- (6.05834808703073,-0.8905940416845036);
\draw (6.05834808703073,-0.8905940416845036)-- (6.228334966136731,-0.7968081773501582);
\draw [shift={(6.268947884403487,2.106519246281936)},line width=2.pt]  plot[domain=4.690375628370565:4.988939260092265,variable=\t]({1.*0.6003506774523132*cos(\t r)+0.*0.6003506774523132*sin(\t r)},{0.*0.6003506774523132*cos(\t r)+1.*0.6003506774523132*sin(\t r)});
\draw [line width=2.pt] (6.764002582802883,1.6447497983062986)-- (6.432694823331695,1.5295860380958797);
\draw [line width=2.pt] (6.878419239963564,1.9790956941175675)-- (6.65331050912446,2.822230213260391);
\draw (7.255404858105066,1.975413364799805) node[anchor=north west] {$\ell_1$};
\draw (4.4563424859384,2.6) node[anchor=north] {$\ell_2$};
\draw (4.5,1.2) node[anchor=north west] {$\ell_3$};
\draw (3.3771859087175162,-1.0035084369036857) node[anchor=north west] {$\ell_5$};
\draw (6.3,-1.) node[anchor=north west] {$\ell_4$};
\draw (4.6,-.6) node[anchor=north west] {$B_3$};
\draw (7.335264287641879,0.7726034297723576) node[anchor=north west] {$B_3$};
\draw (3.1,0.46) node[anchor=north west] {$B_1$};
\draw (3.3,1.7358350761893612) node[anchor=north west] {$B_1$};
\draw (5.602946349235588,0.5590203571973904) node[anchor=north west] {$B_2$};
\draw (5.572735493540809,2.3) node[anchor=north west] {$B_4$};

\node [draw=white, fill=white,rectangle] (a) at (5,-2.8) {(a)};
\end{tikzpicture}
\hspace{2cm}
\begin{tikzpicture}[scale=0.8]
\tikzstyle{every node}=[circle, draw=black, very thick, fill=white, inner sep=0pt, minimum width=4pt];

\draw[very thick](-3,1.46)--(2.5,1.1);
\draw[very thick](-2.8,0.35)--(2.9,-0.5);
\draw[very thick](-3.3,-0.5)--(2.5,-2.8);
\draw[very thick](-1.4,-2.3)--(3.5,-0.1);
\draw[very thick](1.1,2.5)--(1.7,-2.2);

\path (-0.4,2) node (u4l) {} ;
\path (0.2,1.7) node (u4r) {} ;
\draw[very thick](u4l)--(u4r);
\node [draw=white, fill=white,rectangle] (a) at (-0.8,2.3) {$B_4$};
\node [draw=white, fill=white,rectangle] (a) at (1.4,2.7) {$\ell_1$};

\draw[dashed](u4l)--(1,2.2);
\draw[dashed](u4r)--(1.1,1.82);
\draw[dashed](u4l)--(-0.5,1.3);
\draw[dashed](u4r)--(0.15,1.3);

\path (-1.3,0.8) node (u1l) {} ;
\path (-0.7,0.7) node (u1r) {} ;
\draw[very thick](u1l)--(u1r);
\node [draw=white, fill=white,rectangle] (a) at (-0.2,0.7) {$B_1$};
\node [draw=white, fill=white,rectangle] (a) at (-3.3,1.6) {$\ell_2$};

\draw[dashed](u1l)--(-1.25,1.3);
\draw[dashed](u1r)--(-0.64,1.3);
\draw[dashed](u1l)--(-1.4,0.2);
\draw[dashed](u1r)--(-0.75,0.2);
\draw[dashed](u1l)--(-2.03,-1);
\draw[dashed](u1r)--(-1.45,-1.2);

\path (-0.4,-0.7) node (u2l) {} ;
\path (0.3,-0.8) node (u2r) {} ;
\draw[very thick](u2l)--(u2r);
\node [draw=white, fill=white,rectangle] (a) at (-0.05,-0.4) {$B_2$};
\node [draw=white, fill=white,rectangle] (a) at (-3.1,0.5) {$\ell_3$};

\draw[dashed](u2l)--(-0.3,0);
\draw[dashed](u2r)--(0.45,-0.1);
\draw[dashed](u2l)--(0,-1.6);
\draw[dashed](u2r)--(0.58,-1.4);

\path (2.67,-1.1) node (u3r) {} ;
\path (2.3,-1.5) node (u3l) {} ;
\draw[very thick](u3l)--(u3r);
\node [draw=white, fill=white,rectangle] (a) at (3.15,-1.1) {$B_3$};
\node [draw=white, fill=white,rectangle] (a) at (-3.1,-0.5) {$\ell_5$};
\node [draw=white, fill=white,rectangle] (a) at (-1.7,-2) {$\ell_4$};

\draw[dashed](u3l)--(1.93,-0.8);
\draw[dashed](u3r)--(2.43,-0.6);
\draw[dashed](u3l)--(1.66,-1.58);
\draw[dashed](u3r)--(1.65,-1.21);
\draw[dashed](u3l)--(1.9,-2.5);
\draw[dashed](u3r)--(2.05,-2.6);

\node [draw=white, fill=white,rectangle] (a) at (0,-2.8) {(b)};

\end{tikzpicture}
\caption{A 4-body sliding pair chain (a) that is modelled as a point-line framework (b).
A dashed line between a point and a line indicates a point-line distance constraint, and a solid line between two points indicates  a point-point distance constraint.}
\label{fig:DCa}
\end{center}
\end{figure}


\section{Rigidity preserving transformations}
In this section we explain how the rigidity of point-hyperplane frameworks
 is related to the rigidity of
bar-joint frameworks on the sphere or in Euclidean space by using a
rigidity preserving transformation.

We use $\mathbb{R}^d$ to denote $d$-dimensional Euclidean space
equipped with the standard inner product $\langle \cdot,
\cdot\rangle$ and $\mathbb{S}^d$ to denote the unit $d$-dimensional
sphere centered at the origin, and consider
$\mathbb{S}^d\subset \mathbb{R}^{d+1}$. Let ${\bf e}\in
\mathbb{R}^{d+1}$ be the vector whose last coordinate is one and the
others are equal to zero, and let $\mathbb{A}^{d}=\{x\in
\mathbb{R}^{d+1} : \langle x, {\bf e}\rangle=1\}$ be the hyperplane
of $\mathbb{R}^{d+1}$ with ${\bf e}\in \mathbb{A}^{d}$ and with
normal ${\bf e}$. We also use $\mathbb{S}^{d}_{>0}=\{x\in
\mathbb{S}^{d} : \langle x, {\bf e}\rangle> 0\}$,
$\mathbb{S}^{d}_{\geq 0}=\{x\in \mathbb{S}^{d} : \langle x, {\bf
e}\rangle\geq 0\}$ and put
$\mathbb{S}^{d}_{<0}=\mathbb{S}^{d}\setminus \mathbb{S}^{d}_{\geq
0}$. The  {\em equator} of $\mathbb{S}^{d}$ is defined to be
$\mathbb{S}^d_{\geq 0} \setminus \mathbb{S}^d_{>0}$ and is denoted
by $Q$. In the following discussion, the last coordinate in
$\mathbb{R}^{d+1}$ will have a special role (as one may expect from
the definitions of $\mathbb{A}^d$ and $\mathbb{S}^d_{>0}$). Hence we
sometimes refer to a coordinate of a point in $\mathbb{R}^{d+1}$ as
a pair $(x,x')\in \mathbb{R}^d\times \mathbb{R}$, where $x'$ denotes
the last coordinate. For example, a point in $\mathbb{A}^d$ is
denoted by $(x,1)$ with $x\in \mathbb{R}^d$.


\subsection{Euclidean space vs~spherical space}\label{euclid/sphere}
It is a classical fact that there is a one-to-one correspondence
between frameworks in $\mathbb{R}^d$ and those in
$\mathbb{S}^d_{>0}$ at the level of infinitesimal motions.
 Since the
transformation between these two spaces is the starting point of our
study, we first give a detailed description of this transformation.

By a {\em framework in a space $M$} we mean a pair $(G,p)$ of an
undirected finite graph $G=(V,E)$ and a map $p:V\rightarrow M$. The
most widely studied examples are frameworks $(G,p)$ in
$\mathbb{R}^d$, where $p$ is a map from $V$ to $\mathbb{R}^d$. In
this space we are interested in whether there is a different
framework (up to congruences) in some neighborhood of $p$ satisfying the same system of length
constraints:
\[
\|p_i-p_j\|=\text{const} \qquad (ij\in E).
\]
A common strategy to answer this question is to take the derivative of the square of each length constraint to get the first-order length constraint,
\begin{equation}
\label{eq:length_inf}
\langle p_i-p_j, \dot{p}_i-\dot{p}_j\rangle=0 \qquad (ij\in E),
\end{equation}
and then  check the dimension of the solution space with variables
$\dot{p}$. We say that $\dot{p}:V\rightarrow \mathbb{R}^d$ is an
{\em infinitesimal motion} of $(G,p)$ if $\dot{p}$ satisfies
(\ref{eq:length_inf}), and $(G,p)$ is called {\em infinitesimally
rigid} if the dimension of the space of infinitesimal motions of
$(G,p)$ is equal to ${d+1\choose 2}$ (assuming that the points
$p(V)$ affinely span $\mathbb{R}^d$). 

Less well-known but still widely appearing models of frameworks are
those in $\mathbb{S}^d$. In $\mathbb{S}^d$ the spherical distance between
two points is determined by their inner product. Hence we are
interested in the solutions to the system of inner product
constraints:
\[
\langle p_i, p_j\rangle=\text{const} \qquad (ij\in E).
\]
Since $p_i$ is constrained to be on $\mathbb{S}^d$, we also have the
extra constraints
\[
\langle p_i, p_i\rangle=1 \qquad (i\in V).
\]
Again, taking the derivative, we can obtain the system of
first-order inner product constraints:
\begin{align}
\label{eq:inner_inf}
\langle p_i, \dot{p}_j\rangle+\langle p_j, \dot{p}_i\rangle&=0 \qquad (ij\in E) \\
\langle p_i, \dot{p}_i\rangle&=0 \qquad (i\in V). \label{eq:scale}
\end{align}
A map $\dot{p}:V\rightarrow \mathbb{R}^{d+1}$ is  said to be an {\em
infinitesimal motion} of $(G,p)$ if it satisfies this system of
linear constraints, and the framework $(G,p)$ is {\em
infinitesimally rigid} if the dimension of its space of
infinitesimal motions is equal to ${d+1\choose 2}$ (assuming the
points $p(V)$ linearly span $\mathbb{R}^{d+1}$).
 For each $x\in \mathbb{S}^d$, let $$T_x\mathbb{S}^d=\{m\in \mathbb{R}^{d+1}\mid \langle x,m\rangle=0\}$$ be
 the tangent hyperplane at $x$.
 Then we may give an equivalent definition for an infinitesimal motion of $(G,p)$ as  a map
$ i\mapsto \dot{p}_i\in T_{p_i}\mathbb{S}^d$
 which satisfies (\ref{eq:inner_inf}) for all $i\in V$.

In order to relate the rigidity models in $\mathbb{R}^d$ and $\mathbb{S}^d$, a key step is to
identify $\mathbb{R}^d$ with the hyperplane $\mathbb{A}^d$ in
$\mathbb{R}^{d+1}$. For a framework $(G,p)$ in $\mathbb{A}^{d}$, we
define an infinitesimal motion as a map
$i\mapsto \dot{p}_i\in T_{p_i}\mathbb{A}^d$ satisfying
(\ref{eq:length_inf}), where
$$T_{x}\mathbb{A}^d=\{m\in \mathbb{R}^{d+1}\mid \langle {\bf e}, m\rangle=0\}$$ for all $x\in
\mathbb{A}^d$.  Then the space of infinitesimal motions $\dot p$ of a framework $(G,p)$ in $\mathbb{R}^d$ coincides with the space of
infinitesimal motions $\dot{\hat p}$ of the framework $(G, \hat p)$ in
$\mathbb{A}^{d}$,  when we take $\hat p_i=(p_i,1)$ and $ \dot {\hat p}_i=(\dot
p_i,0)$ for all $i \in V$. Hence in the subsequent discussion we may
consider the infinitesimal rigidity of frameworks in $\mathbb{A}^d$
rather than  $\mathbb{R}^d$.

We can now describe the rigidity preserving transformation from
${\mathbb S}^d$ to ${\mathbb A}^d$. Let
$\phi:\mathbb{A}^d\rightarrow \mathbb{S}^d_{>0}$ be the central
projection, that is,
\begin{equation}
\label{eq:phi1}
\phi(x)=\frac{x}{\|x\|}\qquad (x\in \mathbb{A}^d).
\end{equation}
For each $x\in \mathbb{A}^d$, define $\psi_x:T_x\mathbb{A}^d\rightarrow T_{\phi(x)}\mathbb{S}^d$ by
\[
\psi_x(m)=\frac{m-\langle m, x\rangle {\bf e}}{\|x\|} \qquad (m\in T_x\mathbb{A}^d).
\]
The image of $\psi_x$ indeed lies in $T_{\phi(x)}\mathbb{S}^d$ because
\[
\langle \phi(x), \psi_x(m)\rangle=
\frac{\langle x, m-\langle m, x\rangle {\bf e} \rangle}{\|x\|}
=\frac{\langle x, m\rangle- \langle m, x\rangle}{\|x\|}=0
\]
where $\langle x, \langle m, x\rangle {\bf e} \rangle=\langle m, x\rangle$ follows from the fact that
the last coordinate of $x\in \mathbb{A}^d$ is equal to one.

Given a framework $(G,p)$ in $\mathbb{A}^d$ and an infinitesimal
motion $\dot{p}$ of $(G,p)$, a simple calculation shows that
\[
\langle \phi(p_i),
\psi_{p_j}(\dot{p}_j)\rangle+\langle \phi(p_j),
\psi_{p_i}(\dot{p}_i)\rangle=-\frac{\langle p_i-p_j,
\dot{p}_i-\dot{p}_j\rangle}{\|p_i\|\|p_j\|}= 0
\]
for all $ij\in E$, and hence $\psi(\dot p):=(\psi_{p_i}(\dot{p}_i))_{i\in V}$ is
an infinitesimal motion of $(G,\phi\circ p)$ in $\mathbb{S}^d$.
Since $\psi_x$ is bijective for all $x\in \mathbb{A}^d$, it is invertible. This implies that $\psi$ is invertible and this gives us
an isomorphism between the spaces of infinitesimal motions of
$(G,p)$ and $(G,\phi\circ p)$ (see also Figure~\ref{fig:transfer}). In particular, we have the following
result discussed in  \cite{Izmestiev,SaliolaWh,BSWWSphere}.

\begin{proposition}
\label{prop:sdrd}
A bar-joint framework $(G,p)$ is infinitesimally rigid in
$\mathbb{A}^d$ if and only if $(G,\phi \circ p)$ is  infinitesimally
rigid in $\mathbb{S}_{>0}^d$.
\end{proposition}

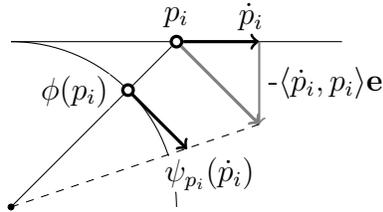
\begin{figure}[htp]
\begin{center}
\begin{tikzpicture}[scale=1.1]
\tikzstyle{every node}=[circle, draw=black, very thick, fill=white, inner sep=0pt, minimum width=4pt];

\draw (0,2) arc (90:0:2cm and 2cm);
\draw (0,2)--(4,2);

\path (2,2) node (pi) {} ;
\node [draw=white, fill=white,rectangle] (a) at (2,2.3) {$p_i$};
\draw[very thick,->](pi)--(3,2);
\node [draw=white, fill=white,rectangle] (a) at (2.9,2.3) {$\dot p_i$};

\filldraw[draw=black] (0,0) circle (1pt);
\draw(0,0)--(pi);
\draw[very thick,->](1.414,1.414)--(2.128,0.7);

\node [draw=white, fill=white,rectangle] (a) at (2.4,0.4) {$\psi_{p_i}(\dot{p}_i)$};

\draw[thick,gray,->](3,2)--(3,1);
\node [draw=white, fill=white,rectangle] (a) at (3.8,1.5) {-$\langle \dot p_i, p_i\rangle {\bf e}$};

\draw[very thick, gray,->](pi)--(3,1);

\path (1.414,1.414) node (psipi) {} ;
\node [draw=white, fill=white,rectangle] (a) at (0.77,1.4) {$\phi(p_i)$};
\draw[dashed](0,0)--(3,1);

\end{tikzpicture}
\caption{Transfer of infinitesimal motions between $\mathbb{A}^d$ and $\mathbb{S}_{>0}^d$.}
\label{fig:transfer}
\end{center}
\end{figure}

In the following, we will extend the correspondence between infinitesimally rigid frameworks in $\mathbb{R}^d$ and $\mathbb{S}^d$ given in Proposition~\ref{prop:sdrd} further by allowing points to lie on the equator of the sphere.
Note that, in the transformation described above, a point on the equator of $\mathbb{S}^d$ corresponds to a `point at infinity' in $\mathbb{A}^d$.


\subsection{Point-hyperplane vs~bar-joint}\label{subsec:ptln}

The frameworks considered in Section \ref{euclid/sphere} model a
structure consisting of rigid bars and joints. Such frameworks are
usually called  {\em bar-joint frameworks}.
A different kind of framework
consisting of points and lines in ${\mathbb R}^2$ mutually linked by
distance or angle constraints (see
Figure~\ref{fig:DCa}(b)
for example), usually referred to as {\em
point-line frameworks}, were introduced in \cite{Owen}. A combinatorial characterization for generic rigidity of such frameworks was recently provided in \cite{JO}.  We will
consider the $d$-dimensional generalisation of these frameworks and
refer to them as {\em point-hyperplane frameworks}. We will use the
rigidity preserving transformation given in Section~\ref{euclid/sphere} to establish an equivalence (at the level of infinitesimal rigidity) between a
point-hyperplane framework in $\mathbb{R}^d$
 and a bar-joint framework in $\mathbb{R}^d$ in
which a given set of joints lie on the same hyperplane. The idea is to use this
 transformation to show that these frameworks are equivalent to a  pair of congruent spherical
 frameworks.

Formally, we define a {\em point-hyperplane framework} in $\mathbb{R}^d$
to be a triple
$(G, p, \ell)$ where $G=(V_P\cup V_L,
E)$ is a {\em point-hyperplane graph}, i.e. a graph $G$ in which the vertices have been partitioned into two sets $V_P, V_L$ corresponding to points and hyperplanes, respectively, and  each edge in $E$ indicates a point-point distance constraint, a
point-hyperplane distance constraint, or a hyperplane-hyperplane
angle constraint. Thus the edge set $E$ is partitioned into three subsets $E_{PP},E_{PL},E_{LL}$ according to the types of end-vertices of each edge.
The point-configuration and the line-configuration are specified by
$p:V_P\to \mathbb{R}^d$, and $\ell=(a,r):V_L\to\mathbb{S}^{d-1}\times \mathbb{R}$, where the hyperplane associated to each $j\in V_L$ is given by
$\{x\in \mathbb{R}^d\,:\,\langle x,a_j\rangle +r_j=0\}$.
For $i\in V_P$ and $j,k\in V_L$, the distance between the point $p_i$ and the hyperplane $\ell_j$ is equal to $|\langle p_i, a_j\rangle+r_j|$, and  the angle between the two hyperplanes $\ell_j,\ell_k$  is determined by
$\langle a_j, a_k\rangle$. Hence the system of constraints can be
written as
\begin{align}
\label{eq:line_const1_euc}
\| p_i-p_j\|&=\text{const} && (ij\in E_{PP}) \\
|\langle p_i, a_j\rangle+r_j|&=\text{const} && (ij\in E_{PL}) \label{eq:line_const2_euc}\\
\langle a_i, a_j\rangle&=\text{const} && (ij\in E_{LL}). \label{eq:line_const3_euc}
\end{align}
Since  $a_j\in\mathbb{S}^{d-1}$,
we also have the constraint
\begin{align*}
\langle a_i, {a}_i\rangle&=1 & (i\in V_L).
\end{align*}
Taking the derivative we get the system of first order constraints
\begin{align}
\langle p_i - p_j, \dot{p}_i-\dot{p}_j\rangle&=0 && (ij\in E_{PP}) \label{eq:line_inf1_euc} \\
\langle p_i, \dot{a}_j\rangle+\langle \dot{p}_i, a_j\rangle +\dot{r}_j&=0 && (ij\in E_{PL}) \label{eq:line_inf2_euc}\\
\langle a_i, \dot{a}_j\rangle+\langle \dot{a}_i, a_j\rangle&=0 && (ij\in E_{LL}) \label{eq:line_inf3_euc}\\
\langle a_i, \dot{a}_i\rangle&=0 && (i\in V_L).
\label{eq:a_inf_euc}
\end{align}

A map $(\dot{p},\dot\ell)$ is  said to be an {\em
infinitesimal motion} of $(G,p,\ell)$ if it satisfies this system of
linear constraints, and
$(G,p,\ell)$ is {\em infinitesimally
rigid} if the dimension of the space of its infinitesimal motions is
equal to ${d+1\choose 2}$, assuming the points $p(V_P)$ and hyperplanes $\ell(V_L)$ affinely span $\mathbb{R}^d$.\footnote{This formulation is slightly different to the formulation for point-line frameworks given in \cite{JO}, but they are easily seen to be  equivalent.}

In order to use the
rigidity preserving transformation from Section
\ref{euclid/sphere}, we will first translate the point-hyperplane framework $(G,p,\ell)$ to a point-hyperplane framework $(G,\hat{p},\ell)$ in affine space
$\mathbb{A}^d$ by taking $\hat{p}_i=(p_i,1)$ for all $i\in V_p$. The system of constraints (\ref{eq:line_inf1_euc})-(\ref{eq:a_inf_euc}) then becomes:

\begin{align}
\langle {\hat{p}}_i - {\hat{p}}_j, \dot{\hat{p}}_i-\dot{\hat{p}}_j\rangle&=0 && (ij\in E, i, j\in V_P) \label{eq:line_inf1} \\
\langle \hat{p}_i, \dot{\ell}_j\rangle+\langle \dot{\hat{p}}_i, \ell_j\rangle&=0 && (ij\in E, i\in V_P, j\in V_L) \label{eq:line_inf2}\\
\langle a_i, \dot{a}_j\rangle+\langle \dot{a}_i, a_j\rangle&=0 && (ij\in E, i, j\in V_L) \label{eq:line_inf3}\\
\langle \dot{\hat{p}}_i,\bf{e}\rangle&=0 &&(i\in V_P) \\
\langle a_i, \dot{a}_i\rangle&=0 && (i\in V_L).
\label{eq:a_inf}
\end{align}

We now relate this system of linear equations with that for bar-joint frameworks.
We first observe
that $r_j$ does not appear in (\ref{eq:line_inf2}) because
$\dot{\hat{p}}_i\in T_{\hat{p}_i}\mathbb{A}^d$ (and hence the last coordinate of
$\dot{\hat{p}}_i$ is equal to zero). This implies that the last coordinate
of $\ell_j$ is not important when analyzing the infinitesimal
rigidity of $(G,\hat{p},\ell)$, and we may always assume that $\ell$ is a
map with $\ell:V_L\rightarrow \mathbb{S}^{d-1}\times \{0\}$. Under
this assumption, we can regard each $\ell_i$ as a point on the
equator $Q$ of $\mathbb{S}^{d}$ by identifying
$\mathbb{S}^{d-1}\times \{0\}$ with $Q$.
Hence (\ref{eq:a_inf}) can be written as $\langle \ell_j, \dot{\ell}_j\rangle=0$, i.e.
$\dot{\ell}_j\in T_{\ell_j}\mathbb{S}^d$ for all $j\in V_L$, and
(\ref{eq:line_inf3}) gives
\[
\langle \ell_i, \dot{\ell}_j\rangle+\langle \dot{\ell}_i, \ell_j\rangle=0
\]
for all $ij\in E$ with $i, j\in V_L$.
We have already seen that (\ref{eq:line_inf1}) can be rewritten as
\[
\langle \phi(\hat{p}_i),
\psi_{\hat{p}_j}(\dot{\hat{p}}_j)\rangle+\langle \phi(\hat{p}_j),
\psi_{\hat{p}_i}(\dot{\hat{p}}_i)\rangle=\frac{\langle \hat{p}_i-\hat{p}_j,
\dot{\hat{p}}_i-\dot{\hat{p}}_j\rangle}{\|\hat{p}_i\|\|\hat{p}_j\|}= 0
\]
for all $ij\in E$ with $i,j\in V_P$. A similar calculation shows that
(\ref{eq:line_inf2}) can be rewritten as
\[
\langle \phi(\hat{p}_i), \dot{\ell}_j\rangle+\langle \psi_{\hat{p}_i}(\dot{\hat{p}}_i), \ell_j \rangle=
\frac{\langle \hat{p}_i, \dot{\ell}_j\rangle+\langle \dot{\hat{p}}_i, \ell_j\rangle}{\|\hat{p}_i\|}=0
\]
for all $ij\in E$ with $i\in V_P$ and $j\in V_L$.

These equations imply that $(\dot{\hat{p}}, \dot{\ell})$ is an infinitesimal motion of $(G,\hat{p},\ell)$ if and only if
$\dot{q}$ is  an infinitesimal motion of $(G,q)$, where
$(G,q)$ is the bar-joint framework in $\mathbb{S}_{\geq 0}^d$ given by
\begin{equation}
\label{eq:phi2}
q_i=\begin{cases}
\phi(\hat{p}_i) & (i\in V_P) \\
(a_i,0) & (i\in V_L),
\end{cases}
\end{equation}
and
$\dot{q}_i\in T_{q_i}\mathbb{S}^d$ is given by
\begin{equation}
\label{eq:phi3}
\dot{q}_i=\begin{cases}
\psi_{\hat{p}_i}(\dot{\hat{p}}_i) & (i\in V_P) \\
\dot{\ell}_i & (i\in V_L).
\end{cases}
\end{equation}
Since each $\psi_x$ is bijective and hence invertible, this gives us an isomorphism between the spaces of infinitesimal motions of
$(G,\hat{p},\ell)$ and $(G, q)$. In particular, if we denote the map $q$ given in (\ref{eq:phi2}) by $\phi\circ
(\hat{p},\ell)$, then
 we  have the
following result.

\begin{theorem}\label{thm:transfer}
Let $(G,\hat{p},\ell)$ be a point-hyperplane framework with $G=(V_P\cup V_L, E)$,
$\hat{p}:V_P\rightarrow \mathbb{A}^d$ and $\ell=(a,r): V_L\rightarrow \mathbb{S}^{d-1}\times \mathbb{R}$.
Let $(G,
\phi\circ (\hat{p},\ell))$ be the bar-joint framework in $\mathbb{S}^d_{\geq 0}$
obtained  by central projection of each $\hat{p}_i \ (i\in V_P)$ and by
regarding each hyperplane $\ell_i=(a_i,r_i) \ (i\in V_L)$ as the
point $(a_i,0)$ on the equator of $\mathbb{S}^d$.
Then $(G,\hat{p},\ell)$ is infinitesimally rigid if and only if $(G, \phi\circ (\hat{p},\ell))$
is infinitesimally rigid.
\end{theorem}

\begin{figure}
\begin{minipage}{0.49\textwidth}
\centering
\begin{tikzpicture}[scale=1.3]
\tikzstyle{every node}=[circle, draw=black, very thick, fill=white, inner sep=0pt, minimum width=4pt];
\node [draw=white, fill=white,rectangle] (a) at (0,1.73) {$v_1$};
\node [draw=white, fill=white,rectangle] (a) at (1.4,0) {$v_2$};
\node [draw=white, fill=white,rectangle] (a) at (-1.3,0) {$v_3$};

\node [draw=white, fill=white,rectangle] (a) at (0.1,0.35) {$u_2$};
\node [draw=white, fill=white,rectangle] (a) at (-0.7,0.6) {$u_1$};
\node [draw=white, fill=white,rectangle] (a) at (-0.4,-0.2) {$u_4$};
\node [draw=white, fill=white,rectangle] (a) at (0,-0.95) {$u_3$};

\draw[very thick] (-2.3,1.5)--(2.3,1.5);
\draw[very thick] (0.2,-1.7)--(-2,1.7);
\draw[very thick] (-0.2,-1.7)--(2.2,1.7);
\path (-0.25,0.35) node (u2) {} ;
\path (-0.5,0.4) node (u1) {} ;
\path (-0.1,-0.2) node (u4) {} ;
\path (-0.1,-0.7) node (u3) {} ;

\draw[very thick](u2)--(u1);
\draw[very thick](u3)--(u4);

\draw[dashed](u1)--(-0.5,1.5);
\draw[dashed](u2)--(-0.25,1.5);
\draw[dashed](u1)--(0.2*3.4,-2.4*0.2)+(u1);
\draw[dashed](u2)--(0.29*3.4-0.25,-2.4*0.29+0.35);

\draw[dashed](u3)--(0.2,-0.93);
\draw[dashed](u4)--(0.5,-0.62);
\draw[dashed](u3)--(-0.3,-0.86);
\draw[dashed](u4)--(-0.5,-0.46);

\draw[thin] (-1.5,1) arc (315:346:1cm);

\end{tikzpicture}
\par
(a)
\end{minipage}
\begin{minipage}{0.48\textwidth}
\centering	
\begin{tikzpicture}[scale=2.5]
\tikzstyle{whitenode} = [circle, draw=black, very thick, fill=black,
inner sep=0pt, minimum width=6pt]
\tikzstyle{thickedge} = [line width=2]
\tikzstyle{blacknode} = [circle, draw=black, very thick, fill=white,
inner sep=0pt, minimum width=6pt]
\tikzset{partial ellipse/.style args={#1:#2:#3}{insert path={+ (#1:#3)
arc (#1:#2:#3)} }}
\draw[thin] (0,0) circle [radius=1];
\draw[thick,dashed,rotate= 0.000000000000000 ] (1, 0) arc ( 0:360:1 and
0.125333233564304 );
\draw[dotted,rotate= 107.075953377807 ] (1, 0) arc ( 0:360:1 and
0.800894826150018 );
\draw[thickedge,rotate= 107.075953377807 ]  (0,0) [partial ellipse=
54.3415707330460 : 38.2336881822120 :1 and  0.800894826150018 ];
\draw[dotted,rotate= 167.795560976190 ] (1, 0) arc ( 0:360:1 and
0.364111601297783 );
\draw[thickedge,rotate= 167.795560976190 ]  (0,0) [partial ellipse=
-31.4928780356244 : 24.0572199764221 :1 and  0.364111601297783 ];
\draw[dotted,rotate= 334.166129934218 ] (1, 0) arc ( 0:360:1 and
0.0330314272341446 );
\draw[thickedge,rotate= 334.166129934218 ]  (0,0) [partial ellipse=
209.130388779898 : 282.490443307557 :1 and  0.0330314272341446 ];
\draw[dotted,rotate= 157.902530370385 ] (1, 0) arc ( 0:360:1 and
0.623910992999180 );
\draw[thickedge,rotate= 157.902530370385 ]  (0,0) [partial ellipse=
-28.3070718656813 : 29.0661904277877 :1 and  0.623910992999180 ];
\draw[dotted,rotate= 319.848993617484 ] (1, 0) arc ( 0:360:1 and
0.0227803291180145 );
\draw[thickedge,rotate= 319.848993617484 ]  (0,0) [partial ellipse=
158.242508482675 : 77.4357317860857 :1 and  0.0227803291180145 ];
\draw[dotted,rotate= 139.400842554892 ] (1, 0) arc ( 0:360:1 and
0.829924012947028 );
\draw[thickedge,rotate= 139.400842554892 ]  (0,0) [partial ellipse=
247.198324640423 : 280.575530054147 :1 and  0.829924012947028 ];
\draw[dotted,rotate= 209.460073108178 ] (1, 0) arc ( 0:360:1 and
0.197081750929911 );
\draw[thickedge,rotate= 209.460073108178 ]  (0,0) [partial ellipse=
148.360851935711 : 95.5450757632254 :1 and  0.197081750929911 ];
\draw[dotted,rotate= 14.3065645525702 ] (1, 0) arc ( 0:360:1 and
0.232188684270779 );
\draw[thickedge,rotate= 14.3065645525702 ]  (0,0) [partial ellipse=
31.9196311215272 : -35.3885493716870 :1 and  0.232188684270779 ];
\draw[dotted,rotate= 250.256428177342 ] (1, 0) arc ( 0:360:1 and
0.211955522827368 );
\draw[thickedge,rotate= 250.256428177342 ]  (0,0) [partial ellipse=
145.864636588848 : 86.8463574357333 :1 and  0.211955522827368 ];
\draw[dotted,rotate= 34.1146535609375 ] (1, 0) arc ( 0:360:1 and
0.581639756347941 );
\draw[thickedge,rotate= 34.1146535609375 ]  (0,0) [partial ellipse=
42.3896374710858 : -43.8263510112185 :1 and  0.581639756347941 ];
\draw[dotted,rotate= 180.000000000000 ] (1, 0) arc ( 0:360:1 and
0.125333233564304 );
\draw[thickedge,rotate= 180.000000000000 ]  (0,0) [partial ellipse=
337.500000000000 : 214.594757077012 :1 and  0.125333233564304 ];
\node [blacknode] at  (-0.793223236023649, 0.366172682076425)  {};
\node [below] at  (-0.793223236023649, 0.366172682076425)  {$u_{ 1 }$};
\node [blacknode] at  (-0.704451539337632, 0.605323469419715)  {};
\node [above right] at  (-0.704451539337632, 0.605323469419715)  {$u_{ 2 }$};
\node [blacknode] at  (0.792130952195805, 0.328702346435997)  {};
\node [below] at  (0.792130952195805, 0.328702346435997)  {$u_{ 3 }$};
\node [blacknode] at  (0.391557933079670, 0.738877215338514)  {};
\node [left] at  (0.391557933079670, 0.738877215338514)  {$u_{ 4 }$};
\node [whitenode] at  (-0.923879532511287, 0.0479629520098034)  {};
\node [left] at  (-0.963879532511287, 0.0479629520098034)  {$v_{ 1 }$};
\node [whitenode] at  (0.180609147449819, -0.123272122028795)  {};
\node [below] at  (0.180609147449819, -0.143272122028795)  {$v_{ 2 }$};
\node [whitenode] at  (0.823188326350385, 0.0711602520617515)  {};
\node [below] at  (0.823188326350385, 0.0711602520617515)  {$v_{ 3 }$};
\end{tikzpicture}
\par
(b)
\end{minipage}
\begin{minipage}{0.48\textwidth}
\centering
\begin{tikzpicture}[scale=2.5]
\tikzstyle{whitenode} = [circle, draw=black, very thick, fill=black,
inner sep=0pt, minimum width=6pt]
\tikzstyle{thickedge} = [line width=2]
\tikzstyle{blacknode} = [circle, draw=black, very thick, fill=white,
inner sep=0pt, minimum width=6pt]
\tikzset{partial ellipse/.style args={#1:#2:#3}{insert path={+ (#1:#3)
arc (#1:#2:#3)} }}
\draw[thin] (0,0) circle [radius=1];
\draw[thick,dashed,rotate= 0.000000000000000 ] (1, 0) arc ( 0:360:1 and
0.125333233564304 );
\draw[dotted,rotate= 122.716915265821 ] (1, 0) arc ( 0:360:1 and
0.641348139139878 );
\draw[thickedge,rotate= 122.716915265821 ]  (0,0) [partial ellipse=
60.2319089110344 : 44.1240263602004 :1 and  0.641348139139878 ];
\draw[dotted,rotate= 193.418566953353 ] (1, 0) arc ( 0:360:1 and
0.392955946464273 );
\draw[thickedge,rotate= 193.418566953353 ]  (0,0) [partial ellipse=
-46.4079312067082 : 9.14216680533833 :1 and  0.392955946464273 ];
\draw[dotted,rotate= 354.318753757873 ] (1, 0) arc ( 0:360:1 and
0.0649024428689042 );
\draw[thickedge,rotate= 354.318753757873 ]  (0,0) [partial ellipse=
138.132330107733 : 64.7722755800740 :1 and  0.0649024428689042 ];
\draw[dotted,rotate= 189.375903527387 ] (1, 0) arc ( 0:360:1 and
0.670058248262484 );
\draw[thickedge,rotate= 189.375903527387 ]  (0,0) [partial ellipse=
-46.0220082485242 : 11.3512540449447 :1 and  0.670058248262484 ];
\draw[dotted,rotate= 340.876253902859 ] (1, 0) arc ( 0:360:1 and
0.170485419065059 );
\draw[thickedge,rotate= 340.876253902859 ]  (0,0) [partial ellipse=
147.172852181895 : 66.3660754853060 :1 and  0.170485419065059 ];
\draw[dotted,rotate= 180.357101909203 ] (1, 0) arc ( 0:360:1 and
0.889192456471026 );
\draw[thickedge,rotate= 180.357101909203 ]  (0,0) [partial ellipse=
222.933602000983 : 256.310807414707 :1 and  0.889192456471026 ];
\draw[dotted,rotate= 226.194376835047 ] (1, 0) arc ( 0:360:1 and
0.314468353934086 );
\draw[thickedge,rotate= 226.194376835047 ]  (0,0) [partial ellipse=
160.834097578214 : 108.018321405729 :1 and  0.314468353934086 ];
\draw[dotted,rotate= 36.9180030975035 ] (1, 0) arc ( 0:360:1 and
0.159556077272406 );
\draw[thickedge,rotate= 36.9180030975035 ]  (0,0) [partial ellipse=
18.4026790569097 : -48.9055014363045 :1 and  0.159556077272406 ];
\draw[dotted,rotate= 268.694048444564 ] (1, 0) arc ( 0:360:1 and
0.429599570099408 );
\draw[thickedge,rotate= 268.694048444564 ]  (0,0) [partial ellipse=
150.509586651085 : 91.4913074979705 :1 and  0.429599570099408 ];
\draw[dotted,rotate= 60.5900377161802 ] (1, 0) arc ( 0:360:1 and
0.449181849907447 );
\draw[thickedge,rotate= 60.5900377161802 ]  (0,0) [partial ellipse=
29.8384343224067 : -56.3775541598976 :1 and  0.449181849907447 ];
\draw[dotted,rotate= 198.147782083880 ] (1, 0) arc ( 0:360:1 and
0.130141555015497 );
\draw[thickedge,rotate= 198.147782083880 ]  (0,0) [partial ellipse=
351.526219661714 : 228.620976738726 :1 and  0.130141555015497 ];
\node [blacknode] at  (-0.736742278325007, 0.116823634153509)  {};
\node [below] at  (-0.736742278325007, 0.106823634153509)  {$u_{ 1 }$};
\node [blacknode] at  (-0.763657763540379, 0.362614298101536)  {};
\node [above] at  (-0.763657763540379, 0.372614298101536)  {$u_{ 2 }$};
\node [blacknode] at  (0.728354433015321, 0.610225132131541)  {};
\node [above right] at  (0.728354433015321, 0.610225132131541)  {$u_{ 3 }$};
\node [blacknode] at  (0.231265764153486, 0.865392043431125)  {};
\node [left] at  (0.231265764153486, 0.865392043431125)  {$u_{ 4 }$};
\node [whitenode] at  (-0.945855860268814, -0.289845529431057)  {};
\node [left] at  (-0.945855860268814, -0.289845529431057)  {$v_{ 1 }$};
\node [whitenode] at  (0.429935653394980, 0.0162307619247677)  {};
\node [right] at  (0.449935653394980, 0.0162307619247677)  {$v_{ 2 }$};
\node [whitenode] at  (0.597739156692182, 0.298687146738420)  {};
\node [right] at  (0.617739156692182, 0.298687146738420)  {$v_{ 3 }$};
\end{tikzpicture}
\par
(c)
\end{minipage}
\begin{minipage}{0.48\textwidth}
\centering
\begin{tikzpicture}[scale=2.5]
\tikzstyle{whitenode} = [circle, draw=black, very thick, fill=black,
inner sep=0pt, minimum width=6pt]
\tikzstyle{thickedge} = [line width=2]
\tikzstyle{blacknode} = [circle, draw=black, very thick, fill=white,
inner sep=0pt, minimum width=6pt]
\tikzset{partial ellipse/.style args={#1:#2:#3}{insert path={+ (#1:#3)
arc (#1:#2:#3)} }}
\draw[thin] (0,0) circle [radius=1];
\draw[thick,dashed,rotate= 0.000000000000000 ] (1, 0) arc ( 0:360:1 and
0.125333233564304 );
\draw[dotted,rotate= 122.716915265821 ] (1, 0) arc ( 0:360:1 and
0.641348139139878 );
\draw[thickedge,rotate= 122.716915265821 ]  (0,0) [partial ellipse=
60.2319089110344 : 44.1240263602004 :1 and  0.641348139139878 ];
\draw[dotted,rotate= 13.4185669533530 ] (1, 0) arc ( 0:360:1 and
0.392955946464272 );
\draw[thickedge,rotate= 13.4185669533530 ]  (0,0) [partial ellipse=
133.592068793292 : 9.14216680533829 :1 and  0.392955946464272 ];
\draw[dotted,rotate= 354.318753757873 ] (1, 0) arc ( 0:360:1 and
0.0649024428689042 );
\draw[thickedge,rotate= 354.318753757873 ]  (0,0) [partial ellipse=
138.132330107733 : 64.7722755800740 :1 and  0.0649024428689042 ];
\draw[dotted,rotate= 9.37590352738726 ] (1, 0) arc ( 0:360:1 and
0.670058248262484 );
\draw[thickedge,rotate= 9.37590352738726 ]  (0,0) [partial ellipse=
133.977991751476 : 11.3512540449447 :1 and  0.670058248262484 ];
\draw[dotted,rotate= 340.876253902859 ] (1, 0) arc ( 0:360:1 and
0.170485419065059 );
\draw[thickedge,rotate= 340.876253902859 ]  (0,0) [partial ellipse=
147.172852181895 : 66.3660754853060 :1 and  0.170485419065059 ];
\draw[dotted,rotate= 180.357101909203 ] (1, 0) arc ( 0:360:1 and
0.889192456471026 );
\draw[thickedge,rotate= 180.357101909203 ]  (0,0) [partial ellipse=
222.933602000983 : 256.310807414707 :1 and  0.889192456471026 ];
\draw[dotted,rotate= 226.194376835047 ] (1, 0) arc ( 0:360:1 and
0.314468353934086 );
\draw[thickedge,rotate= 226.194376835047 ]  (0,0) [partial ellipse=
160.834097578214 : 108.018321405729 :1 and  0.314468353934086 ];
\draw[dotted,rotate= 36.9180030975035 ] (1, 0) arc ( 0:360:1 and
0.159556077272406 );
\draw[thickedge,rotate= 36.9180030975035 ]  (0,0) [partial ellipse=
18.4026790569097 : -48.9055014363045 :1 and  0.159556077272406 ];
\draw[dotted,rotate= 268.694048444564 ] (1, 0) arc ( 0:360:1 and
0.429599570099408 );
\draw[thickedge,rotate= 268.694048444564 ]  (0,0) [partial ellipse=
150.509586651085 : 91.4913074979705 :1 and  0.429599570099408 ];
\draw[dotted,rotate= 60.5900377161802 ] (1, 0) arc ( 0:360:1 and
0.449181849907447 );
\draw[thickedge,rotate= 60.5900377161802 ]  (0,0) [partial ellipse=
29.8384343224067 : -56.3775541598976 :1 and  0.449181849907447 ];
\draw[dotted,rotate= 18.1477820838796 ] (1, 0) arc ( 0:360:1 and
0.130141555015497 );
\draw[thickedge,rotate= 18.1477820838796 ]  (0,0) [partial ellipse=
-8.47378033828579 : 48.6209767387262 :1 and  0.130141555015497 ];
\node [blacknode] at  (-0.736742278325007, 0.116823634153509)  {};
\node [below] at  (-0.736742278325007, 0.106823634153509)  {$u_{ 1 }$};
\node [blacknode] at  (-0.763657763540379, 0.362614298101536)  {};
\node [left] at  (-0.763657763540379, 0.362614298101536)  {$u_{ 2 }$};
\node [blacknode] at  (0.728354433015321, 0.610225132131541)  {};
\node [above] at  (0.768354433015321, 0.640225132131541)  {$u_{ 3 }$};
\node [blacknode] at  (0.231265764153486, 0.865392043431125)  {};
\node [left] at  (0.211265764153486, 0.865392043431125)  {$u_{ 4 }$};
\node [whitenode] at  (0.945855860268814, 0.289845529431057)  {};
\node [right] at  (0.945855860268814, 0.289845529431057)  {$v_{ 1 }$};
\node [whitenode] at  (0.429935653394980, 0.0162307619247677)  {};
\node [below] at  (0.429935653394980, 0.0162307619247677)  {$v_{ 2 }$};
\node [whitenode] at  (0.597739156692182, 0.298687146738420)  {};
\node [below right] at  (0.597739156692182, 0.298687146738420)  {$v_{ 3 }$};
\end{tikzpicture}
\par
(d)
\end{minipage}
\begin{center}
\begin{tikzpicture}[scale=3/ 6.88625492817783 ]
\tikzstyle{whitenode} = [circle, draw=black, very thick, fill=black,
inner sep=0pt, minimum width=6pt]
\tikzstyle{thickedge} = [line width=2]
\tikzstyle{blacknode} = [circle, draw=black, very thick, fill=white,
inner sep=0pt, minimum width=6pt]
\draw[thick, dashed] (0, 5.88625492817783 ) -- (0,  -7.68625492817783 );
\draw[thickedge] (-6.88625492817783, -2.64145165647675) --
(-4.90673610096510, -0.605594055699498) ;
\draw[thickedge] (-6.88625492817783, -2.64145165647675) --
(0.000000000000000, 0.000000000000000) ;
\draw[thickedge] (-6.88625492817783, -2.64145165647675) --
(-1.33226762955019e-15, -6.56286985303870) ;
\draw[thickedge] (-4.90673610096510, -0.605594055699498) --
(0.000000000000000, 0.000000000000000) ;
\draw[thickedge] (-4.90673610096510, -0.605594055699498) --
(-1.33226762955019e-15, -6.56286985303870) ;
\draw[thickedge] (-1.96113708810441, 0.0640828764676245) --
(-2.83788015192171, 0.0927316730220544) ;
\draw[thickedge] (-1.96113708810441, 0.0640828764676245) --
(-1.33226762955019e-15, -6.56286985303870) ;
\draw[thickedge] (-1.96113708810441, 0.0640828764676245) --
(1.11022302462516e-15, 4.17654222558803) ;
\draw[thickedge] (-2.83788015192171, 0.0927316730220544) --
(-1.33226762955019e-15, -6.56286985303870) ;
\draw[thickedge] (-2.83788015192171, 0.0927316730220544) --
(1.11022302462516e-15, 4.17654222558803) ;
\draw[thickedge] (0.000000000000000, 0.000000000000000) --
(1.11022302462516e-15, 4.17654222558803) ;
\node [blacknode] at  (-6.88625492817783, -2.64145165647675)  {};
\node [below left] at  (-6.88625492817783, -2.64145165647675)  {$u_{ 1 }$};
\node [blacknode] at  (-4.90673610096510, -0.605594055699498)  {};
\node [above left] at  (-4.90673610096510, -0.605594055699498)  {$u_{ 2 }$};
\node [blacknode] at  (-1.96113708810441, 0.0640828764676245)  {};
\node [right] at  (-1.96113708810441, 0.1640828764676245)  {$u_{ 3 }$};
\node [blacknode] at  (-2.83788015192171, 0.0927316730220544)  {};
\node [left] at  (-2.83788015192171, 0.0927316730220544)  {$u_{ 4 }$};
\node [whitenode] at  (0.000000000000000, 0.000000000000000)  {};
\node [right] at  (0.000000000000000, 0.000000000000000)  {$v_{ 1 }$};
\node [whitenode] at  (-1.33226762955019e-15, -6.56286985303870)  {};
\node [right] at  (-1.33226762955019e-15, -6.56286985303870)  {$v_{ 2
}$};
\node [whitenode] at  (1.11022302462516e-15, 4.17654222558803)  {};
\node [right] at  (1.11022302462516e-15, 4.17654222558803)  {$v_{ 3 }$};
\end{tikzpicture}
\par
(e)
\end{center}
\caption{An illustration of the rigidity preserving transformations in Theorems \ref{thm:transfer} and \ref{thm:transfer2}. (a) A point-line framework $(G,\hat{p},\ell)$. (b) The corresponding spherical framework $(G,\phi\circ (\hat{p},\ell))$ in $\mathbb{S}_{\geq 0}^2$ with three points on the equator. The spherical framework in (c) arises from (b) by a small rotation to take points off the equator. An inversion of points in $\mathbb{S}^{d}_{<0}$ then gives (d). Finally in (e) we have a projection to the plane as a bar-joint framework with three collinear points.
}
\label{fig:transfer_sph_ptline}
 \end{figure}
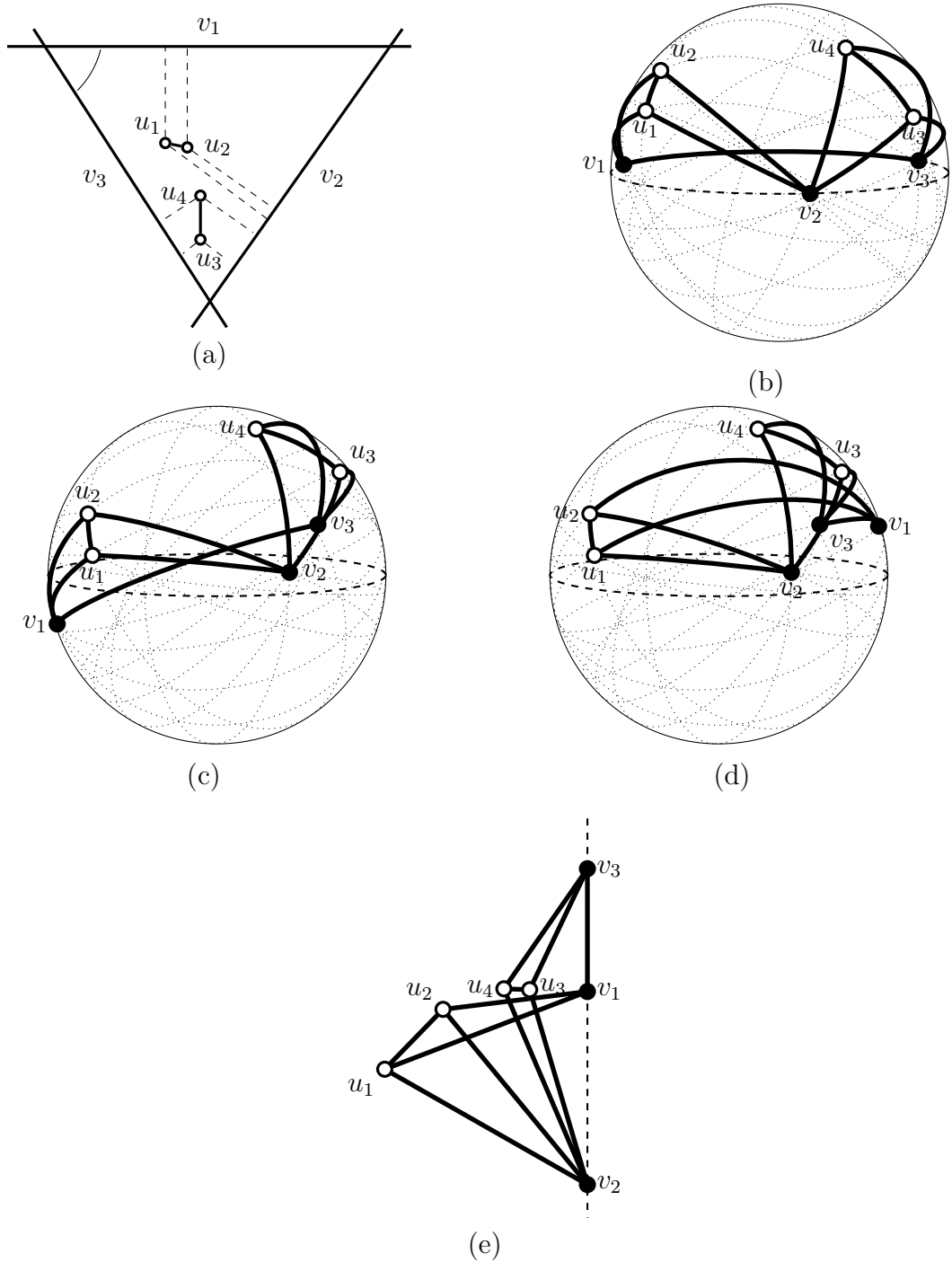

The transformation used in Theorem \ref{thm:transfer} is illustrated in Figure \ref{fig:transfer_sph_ptline}(a), (b).

%

In order to relate $(G, \phi\circ (\hat{p},\ell))$ with a bar-joint
framework in $\mathbb{A}^d$, we further consider transformations
for frameworks in $\mathbb{S}^d$ introduced in \cite{BSWWSphere}.
Given a framework $(G,q)$ in $\mathbb{S}^d$, a {\em
rotation} $\gamma$ is an operator acting on $q$ such that
$(\gamma\circ q)_i=R{q}_i$, for all $i\in V$, for some orthogonal matrix $R$.
Note that $\dot{{q}}$ is an infinitesimal motion of $(G,q)$ if and
only if the map $\gamma \circ \dot{q}$ defined by $(\gamma \circ
\dot{{q}})_i=R\dot{{q}}_i\ (i\in V)$ is an infinitesimal motion of
$(G,\gamma\circ q)$.  In particular, $(G,q)$ is infinitesimally
rigid if and only if  $(G,\gamma\circ q)$ is infinitesimally rigid.

Given a framework $(G,{q})$ in $\mathbb{S}^d$ and $I\subseteq V$, the {\em inversion} $\iota$ (with respect to
$I$) is an operator acting on $q$ such that $(\iota\circ q)_i=-q_i$
if $i\in I$ and $(\iota\circ q)_i=q_i$ otherwise. Note that
$\dot{q}$ is an infinitesimal motion of $(G,q)$ if and only if
$\iota \circ \dot{q}$ defined by $(\iota \circ
\dot{q})_i=-\dot{q}_i\ (i\in I)$ and $(\iota \circ
\dot{q})_i=\dot{q}_i\ (i\in V\setminus I)$ is an infinitesimal
motion of $(G,\iota\circ q)$, which again means that $\iota$
 preserves infinitesimal rigidity.

We shall use an inversion to flip points in $\mathbb{S}_{<0}$ to
$\mathbb{S}_{>0}$ so that a framework $(G,q)$ in $\mathbb{S}^d$ is
transferred to a framework $(G,\iota\circ q)$ in $\mathbb{S}^d_{\geq
0}$. In the following discussion, $\iota$ always refers to such an
operator. Then a framework $(G,q)$ in $\mathbb{S}^d$ can be
transformed to a framework $(G, \iota\circ \gamma\circ q)$ in
$\mathbb{S}_{>0}$ by first applying a rotation $\gamma$ which
moves all points off the equator, and then applying $\iota$ to flip points to $\mathbb{S}_{>0}$.
For a framework in
$\mathbb{S}_{>0}$ we can then use the inverse of $\phi$ to transfer
it to $\mathbb{A}^d$.
An important property of this sequence of transformations is that
point-hyperplane incidence is preserved, i.e. points in $(G,q)$ lie
on a hyperplane in $\mathbb{S}^d$ if and only if the corresponding
points in $(G, \phi^{-1} \circ \iota\circ \gamma\circ p)$ lie on a
hyperplane in $\mathbb{A}^d$. Combining this with
Theorem~\ref{thm:transfer} we have our main result. (See also Figure \ref{fig:transfer_sph_ptline} for an illustration.)

\begin{theorem}\label{thm:transfer2}
Let $(G,\hat{p},\ell)$ be a point-hyperplane framework in $\mathbb{A}^d$ with $G=(V_P\cup
V_L, E)$, $\hat p:V_P\rightarrow \mathbb{A}^d$ and $\ell=(a,r):
V_L\rightarrow \mathbb{S}^{d-1}\times \mathbb{R}$. Let
$(G,\hat{q})$ be the bar-joint framework in $\mathbb{A}^d$
with $\hat{q}= \phi^{-1}\circ \iota\circ \gamma \circ \phi\circ (\hat{p},\ell)$.
Then the points in $\hat{q}(V_L)$ all lie on a hyperplane in
$\mathbb{A}^d$, and $(G,\hat{p},\ell)$ is infinitesimally rigid  if and only if $(G,\hat{q})$ is infinitesimally rigid.
\end{theorem}

Note that the above transformation is reversible, i.e., from a
framework $(G,\hat{q})$ in $\mathbb{A}^d$ with points $\hat{q}(X)$ being on a
hyperplane for $X\subset V$, one obtains a point-hyperplane framework $(G,
\hat{p}, \ell)$ in $\mathbb{A}^d$ with $V_L=X$ and $V_P=V\setminus X$ such
that $(G,\hat{q})$ is infinitesimally rigid if and only if $(G, \hat{p}, \ell)$
is infinitesimally rigid. We can now associate $\mathbb{A}^d$ with
$\mathbb{R}^d$ to obtain the following result.
\begin{corollary}\label{cor:regulartransfer}
Let $G=(V,E)$ be a graph and $X\subseteq V$. Then the following are equivalent:
\begin{itemize}
\item[(a)]  $G$ can be realised as an infinitesimally rigid bar-joint framework in $\mathbb{R}^d$ such that the points assigned to $X$ lie on a hyperplane.
\item[(b)] $G$ can be realised as an infinitesimally rigid point-hyperplane framework in $\mathbb{R}^d$ such that each vertex in $X$ is realised as a hyperplane and each vertex in $V\setminus X$ is realised as a point.
\end{itemize}
\end{corollary}

\subsection{Combinatorial characterization in the plane}\label{subsec:jj}
To see the power of our main theorem, let us consider the case when $d=2$.
In the plane, Jackson and Owen \cite{JO} were able to give a combinatorial characterization of graphs which can be realised as an infinitesimally rigid point-line framework.
Combining this with Corollary~\ref{cor:regulartransfer} we immediately obtain
 the following characterization of graphs which can be realised as infinitesimally rigid bar-joint frameworks in the plane with a given set of collinear points.
This theorem extends a result by Jackson and Jord\'an \cite{JJcol}, where they give a characterization for the case when three specified points are collinear.
We will need the following notation. Given a graph $G=(V,E)$, $X\subseteq V$ and $A\subseteq E$, let $\nu_X(A)$ be the number of vertices of $X$ which are incident to edges in $A$.

\begin{theorem}\label{thm:barjointcol}
Let $G=(V,E)$ be a graph and $X\subseteq V$. Then the following are equivalent:
\begin{itemize}
\item[(a)] $G$ can be realised as an infinitesimally rigid bar-joint framework in $\mathbb{R}^2$ such that the points assigned to $X$ lie on a line.

\item[(b)] $G$ can be realised as an infinitesimally rigid point-line framework in $\mathbb{R}^2$ such that each vertex in $X$ is realised as a line and each vertex in $V\setminus X$ is realised as a point.

\item[(c)] $G$ contains a spanning subgraph $G'=(V,E')$ such that $E'=2|V|-3$ and,
for all $\emptyset\neq A\subseteq E'$ and all partitions $\{A_1,\ldots,A_s\}$ of $A$,
$$|A|\leq \sum_{i=1}^s (2\nu_{V\sm X}(A_i)+\nu_{X}(A_i)-2)+\nu_{X}(A)-1
.$$
\end{itemize}
\end{theorem}

We illustrate this result using the underlying graph $G=(V,E)$ in Figure~\ref{fig:transfer_sph_ptline}, taking $X=\{v_1,v_2,v_3\}$. We have
$|E|=11=2|V|-3$ so the only possible choice for the subgraph $G'$
described in
Theorem~\ref{thm:barjointcol}(c) is $G'=G$. However, if we take $A=E$ and
$A_1$, $A_2$ and  $A_3$ to be the edge-sets induced by
$\{v_1,u_1,u_2,v_2\}$, $\{v_2,u_3,u_4,v_3\}$ and $\{v_1,v_3\}$,
respectively, then $\{A_1,A_2,A_3\}$ partitions $E$, $\nu_{V\sm X}(A_1)=\nu_{V\sm X}(A_2)=2$,
$\nu_{V\sm X}(A_3)=0$ and $\nu_{X}(A_1)=\nu_{X}(A_2)=\nu_{X}(A_3)=2$.
Since $\nu_X(E)=3$, this gives
$$|E|=11>\sum_{i=1}^3 (2\nu_{V\sm X}(A_i)+\nu_{X}(A_i)-2)+\nu_{X}(E)-1=10$$
so the condition in Theorem~\ref{thm:barjointcol}(c)
 fails to hold. Hence $G$ cannot be realised as an
infinitesimally rigid point-line framework in $\mathbb{R}^2$ such
that each vertex in $X$ is realised as a line and each vertex in
$V\setminus X$ is realised as a point, and $G$ cannot be realised as
an infinitesimally rigid bar-joint framework in $\mathbb{R}^2$ such
that the points assigned to $X$ lie on a line. Note however that
every generic realisation of $G$ as a bar-joint framework in
$\mathbb{R}^2$ is infinitesimally rigid by Laman's theorem \cite{Lamanbib}.

As a second example, consider the point-line framework in Figure~\ref{fig:DCa}(b). The underlying point-line graph is given in Figure~\ref{fig:walter}(a), and is shown as a bar-joint framework with collinear points. It has
 $|V_P|=8$ and $|V_L|=5$, and hence we have $2|V_P|+2|V_L|-3=23$.
If we take $X$ to be the set of line vertices $V_L$ and $A_i$ to be the set of edges incident to the body $B_i$ for $i=1,2,3,4$, then
$$\sum_{i=1}^4 (2\nu_{V\sm X}(A_i)+\nu_{X}(A_i)-2)+\nu_{X}(E)-1=22.$$
Since $A_1,A_2,A_3,A_4$ partition $E$, no subset $A$ of $E$ with $|A|=23$ can satisfy Theorem~\ref{thm:barjointcol}(c).

\begin{figure}[htp]
\begin{center}
\begin{tikzpicture}[scale=0.8]
\tikzstyle{every node}=[circle, draw=black, very thick, fill=white, inner sep=0pt, minimum width=4pt];

\draw[thick, dashed](0,0)--(6.5,2.2);

 \filldraw[draw=black] (1,0.35) circle (3pt);
 \filldraw[draw=black] (2.5,0.83) circle (3pt);
 \filldraw[draw=black] (3.6,1.2) circle (3pt);
 \filldraw[draw=black] (4.5,1.53) circle (3pt);
 \filldraw[draw=black] (5.8,1.95) circle (3pt);

\node [draw=white, fill=white,rectangle] (a) at (0.6,0.5) {$\ell_1$};
\node [draw=white, fill=white,rectangle] (a) at (2.3,1.2) {$\ell_2$};
\node [draw=white, fill=white,rectangle] (a) at (3.3,0.8) {$\ell_5$};
\node [draw=white, fill=white,rectangle] (a) at (4.5,1.1) {$\ell_4$};
\node [draw=white, fill=white,rectangle] (a) at (6.2,2.3) {$\ell_3$};

\node [draw=white, fill=white,rectangle] (a) at (1.5,2.7) {$B_3$};
\node [draw=white, fill=white,rectangle] (a) at (4.2,3.5) {$B_2$};
\node [draw=white, fill=white,rectangle] (a) at (1.6,-0.9) {$B_4$};
\node [draw=white, fill=white,rectangle] (a) at (4.5,-1) {$B_1$};

\path (1.1,2) node (b3l) {} ;
\path (2,2.5) node (b3r) {} ;
\draw[very thick](b3l)--(b3r);

\path (4.7,3.5) node (b2l) {} ;
\path (4.9,2.6) node (b2r) {} ;
\draw[very thick](b2l)--(b2r);

\path (1.9,-0.2) node (b4l) {} ;
\path (2.2,-0.9) node (b4r) {} ;
\draw[very thick](b4l)--(b4r);

\path (3.9,-0.7) node (b1l) {} ;
\path (4.9,-0.4) node (b1r) {} ;
\draw[very thick](b1l)--(b1r);

\draw[thick](b3l)--(1,0.35);
\draw[thick](b3l)--(3.6,1.2);
\draw[thick](b3l)--(4.5,1.53);
\draw[thick](b3r)--(1,0.35);
\draw[thick](b3r)--(3.6,1.2);
\draw[thick](b3r)--(4.5,1.53);

\draw[thick](b4l)--(2.5,0.83);
\draw[thick](b4l)--(1,0.35);
\draw[thick](b4r)--(2.5,0.83);
\draw[thick](b4r)--(1,0.35);

\draw[thick](b2l)--(4.5,1.53);
\draw[thick](b2l)--(5.8,1.95);
\draw[thick](b2r)--(4.5,1.53);
\draw[thick](b2r)--(5.8,1.95);

\draw[thick](b1l)--(2.5,0.83);
\draw[thick](b1l)--(3.6,1.2);
\draw[thick](b1l)--(5.8,1.95);
\draw[thick](b1r)--(2.5,0.83);
\draw[thick](b1r)--(3.6,1.2);
\draw[thick](b1r)--(5.8,1.95);

\node [draw=white, fill=white,rectangle] (a) at (3.5,-2.8) {(a)};

\end{tikzpicture}
\hspace{1.2cm}
\begin{tikzpicture}[scale=0.8]
\tikzstyle{every node}=[circle, draw=black, very thick, fill=white, inner sep=0pt, minimum width=4pt];

 \filldraw[draw=black] (1,0.85) circle (3pt);
 \filldraw[draw=black] (3.6,1.7) circle (3pt);
 \filldraw[draw=black] (4.5,2.03) circle (3pt);

\node [draw=white, fill=white,rectangle] (a) at (0.6,1) {$\ell_1$};
\node [draw=white, fill=white,rectangle] (a) at (3.3,1.3) {$\ell_5$};
\node [draw=white, fill=white,rectangle] (a) at (4.5,1.6) {$\ell_4$};

\node [draw=white, fill=white,rectangle] (a) at (1.5,3.2) {$B_3$};

\path (1.1,2.5) node (b3l) {} ;
\path (2,3) node (b3r) {} ;
\draw[very thick](b3l)--(b3r);

\draw[thick](b3l)--(1,0.85);
\draw[thick](b3l)--(3.6,1.7);
\draw[thick](b3l)--(4.5,2.03);
\draw[thick](b3r)--(1,0.85);
\draw[thick](b3r)--(3.6,1.7);
\draw[thick](b3r)--(4.5,2.03);

\path (6.7,3.5) node (b2l) {} ;
\path (6.9,2.6) node (b2r) {} ;
\draw[very thick](b2l)--(b2r);

 \filldraw[draw=black] (6.5,1.53) circle (3pt);
 \filldraw[draw=black] (7.8,1.95) circle (3pt);

\node [draw=white, fill=white,rectangle] (a) at (6.5,1.1) {$\ell_4$};
\node [draw=white, fill=white,rectangle] (a) at (8.2,2.3) {$\ell_3$};

\draw[thick](b2l)--(6.5,1.53);
\draw[thick](b2l)--(7.8,1.95);
\draw[thick](b2r)--(6.5,1.53);
\draw[thick](b2r)--(7.8,1.95);

\node [draw=white, fill=white,rectangle] (a) at (6.2,3.5) {$B_2$};

\filldraw[draw=black] (1,-0.65) circle (3pt);
 \filldraw[draw=black] (2.5,-0.17) circle (3pt);

\node [draw=white, fill=white,rectangle] (a) at (0.6,-0.5) {$\ell_1$};
\node [draw=white, fill=white,rectangle] (a) at (2.3,0.2) {$\ell_2$};

\node [draw=white, fill=white,rectangle] (a) at (1.6,-1.9) {$B_4$};

\path (1.9,-1.2) node (b4l) {} ;
\path (2.2,-1.9) node (b4r) {} ;
\draw[very thick](b4l)--(b4r);

\draw[thick](b4l)--(2.5,-0.17);
\draw[thick](b4l)--(1,-0.65);
\draw[thick](b4r)--(2.5,-0.17);
\draw[thick](b4r)--(1,-0.65);


 \filldraw[draw=black] (4.5,-0.67) circle (3pt);
 \filldraw[draw=black] (5.6,-0.3) circle (3pt);
 \filldraw[draw=black] (7.8,0.45) circle (3pt);

\node [draw=white, fill=white,rectangle] (a) at (4.3,-0.3) {$\ell_2$};
\node [draw=white, fill=white,rectangle] (a) at (6,0) {$\ell_5$};
\node [draw=white, fill=white,rectangle] (a) at (8.2,0.8) {$\ell_3$};

\node [draw=white, fill=white,rectangle] (a) at (6.5,-2.5) {$B_1$};

\path (5.9,-2.3) node (b1l) {} ;
\path (6.9,-1.9) node (b1r) {} ;
\draw[very thick](b1l)--(b1r);

\draw[thick](b1l)--(4.5,-0.67);
\draw[thick](b1l)--(5.6,-0.3);
\draw[thick](b1l)--(7.8,0.45);
\draw[thick](b1r)--(4.5,-0.67);
\draw[thick](b1r)--(5.6,-0.3);
\draw[thick](b1r)--(7.8,0.45);

\node [draw=white, fill=white,rectangle] (a) at (4,-2.8) {(b)};
\end{tikzpicture}
\caption{A bar-joint framework with five collinear points corresponding to the point-line framework in Figure~\ref{fig:DCa}(b) and a partition of the edge set into $A_1, A_2, A_3$ and $A_4$.}
\label{fig:walter}
\end{center}
\end{figure}

\subsection{Rigidity matrices}\label{subsec:rigmat}
It is standard to analyze a linear system by using its matrix representation.
Here we shall present the matrices corresponding to the key linear systems discussed above.

Let $(G,p)$ be a bar-joint framework in $\mathbb{R}^d$ with underlying graph $G=(V,E)$.
 The $|E|\times d|V|$ matrix corresponding to the linear system given in (\ref{eq:length_inf}) is
the well known \emph{rigidity matrix} of $(G,p)$, and has the form
\begin{displaymath} R(G,p)=\bordermatrix{& & & & i & & & & j & & & \cr & & & &  & & \vdots & &  & & &
\cr \{i,j\} & 0 & \ldots &  0 & (p_{i}-p_{j}) & 0 & \ldots & 0 & (p_{j}-p_{i}) &  0 &  \ldots&  0 \cr & & & &  & & \vdots & &  & & &
}
\textrm{,}\end{displaymath}
where, for each edge $\{i,j\}\in E$, $R(G,p)$ has the row with
$(p_{i}-p_{j})_{1},\ldots,(p_{i}-p_{j})_{d}$ in the columns associated with $i$,
$(p_{j}-p_{i})_{1},\ldots,(p_{j}-p_{i})_{d}$ in the columns associated with $j$,
and $0$ elsewhere \cite{WhDiscrete}. It is clear from the discussion in Section~\ref{euclid/sphere}
that the rigidity matrix $R(G,p)$ is the fundamental tool to analyse the
infinitesimal rigidity properties of $(G,p)$.

Similarly, for a spherical  framework $(G,p)$ with $p:V\to\mathbb{S}^d$, we can also write down
 a standard spherical rigidity matrix, as described in \cite{NOP,BSWWSphere}. For our purposes it is more instructive to consider the  $(|E|+|V|)\times(d+1)|V|$ matrix $R_{\mathbb{S}}(G,p)$ corresponding to the linear system given in (\ref{eq:inner_inf}) and (\ref{eq:scale}), which can easily be seen to be row-equivalent to the matrix in \cite{NOP,BSWWSphere}:
\begin{displaymath} R_{\mathbb{S}}(G,p)=\bordermatrix{& & & & i& & & & j & & & \cr
 \{i,j\} & 0 & \ldots &  0 & p_{j} & 0 & \ldots & 0 &p_{i} &  0 &  \ldots&  0 \cr
&&  & & \vdots&&&&\vdots&\cr
i& 0 & \ldots &  0 & p_{i} & 0 & \ldots & 0 &0 &  0 &  \ldots&  0 \cr
&&  & & \vdots&&&&\vdots&\cr
j& 0 & \ldots &  0 &0 & 0 & \ldots & 0 &p_{j} &  0 &  \ldots&  0 \cr
}
\textrm{.}\end{displaymath}

In Section~\ref{subsec:ptln} we showed that infinitesimal rigidity can be transferred between a point-hyperplane framework
in $\mathbb{R}^d$ and a bar-joint framework in $\mathbb{R}^d$ which has a given set of points lying in a hyperplane. We proved this by
showing that this transfer of infinitesimal rigidity works for a point-hyperplane framework $(G,p,\ell)$ in $\mathbb{A}^d$ and the spherical framework
$(G,q)$ in $\mathbb{S}^d$ with $q:V\to \mathbb{S}_{\geq 0}^d$ defined as in (\ref{eq:phi2}). In the following, we present the rigidity matrix
$R_{\mathbb{A}}(G,p,\ell)$ for  $(G,p,\ell)$ in $\mathbb{A}^d$ corresponding to the linear system given in (\ref{eq:line_inf1})-(\ref{eq:a_inf}).

Let the underlying graph of $(G,p,\ell)$ be $G=(V_P\cup V_L,E)$ and let $p:V_P\to \mathbb{A}^d$ and $\ell:V_L\to \mathbb{S}^{d-1}\times \mathbb{R}$ with $\ell_i=(a_i,r_i)$.  Further, let $i,j\in V_P$ and $i',j'\in V_L$. Then $R_{\mathbb{A}}(G,p,\ell)$ is an $(|E|+|V|)\times (d+1)|V|$ matrix of the following form:
\begin{displaymath}\bordermatrix{& & i & & j & &k& & l&\cr
 \{i,j\} & 0 \ldots  0 & (p_i-p_j) & 0 \ldots  0  & (p_j-p_i) & 0 \ldots  0 & 0 &0 \ldots  0 &  0 &0 \ldots  0 \cr
  & & &  & &  \vdots &  & &   &  \cr
\{i,k\} & 0 \ldots  0 &(a_k,0) & 0 \ldots  0  & 0& 0 \ldots  0 & p_i &0 \ldots  0 &  0 &0 \ldots  0 \cr
 & & &  & &  \vdots &  & &   &  \cr
\{k,l\} & 0 \ldots  0 & 0 & 0 \ldots  0  & 0& 0 \ldots  0 & (a_{l},0) &0 \ldots  0 &  (a_{k},0) &0 \ldots  0 \cr
 & & &  & &  \vdots &  & &   &  \cr
i & 0 \ldots  0 & \mathbf{e} & 0 \ldots  0  & 0& 0 \ldots  0 &0 &0 \ldots  0 & 0 &0 \ldots  0 \cr
& & &  & &  \vdots &  & &   &  \cr
j & 0 \ldots  0 & 0 & 0 \ldots  0  & \mathbf{e}& 0 \ldots  0 &0 &0 \ldots  0 & 0 &0 \ldots  0 \cr
& & &  & &  \vdots &  & &   &  \cr
k & 0 \ldots  0 & 0& 0 \ldots  0  & 0& 0 \ldots  0 &(a_{k},0)&0 \ldots  0 & 0 &0 \ldots  0 \cr
& & &  & &  \vdots &  & &   &  \cr
l & 0 \ldots  0 & 0 & 0 \ldots  0  & 0& 0 \ldots  0 &0 &0 \ldots  0 & (a_{l},0) &0 \ldots  0 \cr
& & &  & &  \vdots &  & &   &  \cr
}
\textrm{.}\end{displaymath}
Since  $r_j$ does not appear in (\ref{eq:line_inf2}), the last coordinate
of $\ell_{k}$  in the row corresponding to $\{i,k\}$ may be assumed to be zero, i.e., $\ell_{k}=(a_{k},0)$. (Recall Section~\ref{subsec:ptln}). It then follows from the definition of the maps $\phi$ and $\psi$  in Section~\ref{subsec:ptln} that this matrix is row equivalent to the 
spherical rigidity matrix $R_{\mathbb{S}}(G,q)$, where  $q:V\to \mathbb{S}_{\geq 0}^d$ is defined as in (\ref{eq:phi2}). Note that the rows corresponding to $i$ and $j$ in $R_{\mathbb{A}}(G,p,\ell)$ guarantee that $\dot{p}_i$ and $\dot{p}_j$ lie in $T_{p_i}\mathbb{A}^d$ and $T_{p_j}\mathbb{A}^d$, respectively, for any infinitesimal motion $(\dot{p},\dot{\ell})$ of $(G,p,\ell)$. By row operations the entries $\mathbf{e}$ in those rows are converted to $q_i=\phi(p_i)$ and $q_j=\phi(p_j)$ in $R_{\mathbb{S}}(G,q)$, so that the infinitesimal velocity vectors of $(G,q)$ at $q_i$ and $q_j$ lie in $T_{q_i}\mathbb{S}^d$ and $T_{q_j}\mathbb{S}^d$, respectively.

\section{Connection to scene analysis}\label{sec:scene}
In this section we describe a connection between point-hyperplane frameworks and scene analysis.


A {\em $d$-scene} is a pair consisting of a set of points and a set
of hyperplanes in $\mathbb{R}^d$. By using a bipartite graph
$G=(V_P\cup V_L, E)$
to represent the point-hyperplane incidences (where each vertex in
$V_P$ corresponds to a point, each vertex in $V_L$ to a hyperplane,
and each edge in $E$ to a point-hyperplane incidence), a $d$-scene
can be formally defined as a triple $(G, p, \ell)$ of a bipartite
graph $G$, and maps $p:V_P\rightarrow \mathbb{R}^d$ and $\ell:V_L\to
\mathbb{S}^{d-1}\times \mathbb{R}$, satisfying the
incidence condition
\begin{equation}
\langle p_i, a_j\rangle+r_j=0 \qquad (ij\in E, i\in V_P, j\in V_L),
\end{equation}
taking $\ell_j= (a_j, r_j)$ for all $j\in V_L$. Given the hyperplane normals $(a_j)_{j\in V_L}$, we can always construct
a $d$-scene with these normals by choosing the points $P_i$ to be coincident, i.e. putting
$p_i=t\ (i\in V_P)$ and $r_j=-\langle t, a_j\rangle\ (j\in V_L)$ for some $t\in \mathbb{R}^d$. We will call such a $d$-scene {\em trivial}.

 In the {\em realisation problem} (see \cite{WhDiscrete} for example) we are asked whether there is a {\em non-trivial}
$d$-scene with a given set of hyperplane normals $(a_j)_{j\in V_L}$.
Note that the set of all realisations forms a linear space
whose dimension is at least $d$, with equality if and only if  every realisation is trivial. It follows that the existence of a nontrivial realisation can be checked by determining the
dimension of the solution space of the following linear system of
equations for the variables $x:V_p\rightarrow \mathbb{R}^d$ and
$t:E\rightarrow \mathbb{R}$:
\begin{equation}
\label{eq:parallel} \langle x_i, a_j\rangle+t_j=0 \qquad (ij\in E,
i\in V_P, j\in V_L).
\end{equation}

Now let us return to point-hyperplane frameworks, and consider the
restricted rigidity model when the normal of each hyperplane is
fixed. We can obtain the first order constraints for a fixed-normal point-hyperplane framework
$(G,p,\ell)$ with $G=(V_P\cup V_L, E)$ by setting $\dot{a}_j=0$ in the system (\ref{eq:line_inf1})-(\ref{eq:a_inf}). This gives
\begin{align}
\langle p_i - p_j, \dot{p}_i-\dot{p}_j\rangle&=0 && (ij\in E_{PP}) \label{eq:fixed_inf1} \\
\langle \dot{p}_i, a_j\rangle+\dot{r}_j&=0 && (ij\in E_{PL}) \label{eq:fixed_inf2},
\end{align}
where  $\dot{p}$ and $\dot{r}$ are variables. We say that the
 $(G,p,\ell)$ is {\em infinitesimally fixed-normal rigid} if
the dimension of the space of infinitesimal motions, i.e. the
solution space of the  system of equations (\ref{eq:fixed_inf1})
and (\ref{eq:fixed_inf2}), is equal to $d$.
Note that the system of equations (\ref{eq:fixed_inf2}) depends only
 on the normals $(a_j)_{j\in V_L}$. This implies  that the
infinitesimal fixed-normal rigidity of $(G,p,\ell)$ depends only on
the normals $(a_j)_{j\in V_L}$ when $(G,p,\ell)$ is {\em naturally
bipartite} i.e. when all constraints are point-hyperplane distance
constraints.

Now observe that  (\ref{eq:parallel}) and (\ref{eq:fixed_inf2})
represent exactly the same system of equations by identifying $x$
with $\dot{p}$ and $t$ with $\dot{r}$. This means that, for any
bipartite graph $G=(V_P\cup V_L,E)$ and any fixed set of hyperplane
normals $(a_j)_{j\in V_L}$, every realisation of $G$ as a
$d$-scene with hyperplane normals $(a_j)_{j\in V_L}$ is trivial
if and only if every realisation of $G$ as a naturally bipartite
point-hyperplane framework with hyperplane normals $(a_j)_{j\in V_L}$ is
infinitesimally fixed-normal rigid.

Whiteley \cite{WhHypergraph} gave a combinatorial characterization
of generic $d$-scenes
which have  only trivial realisations.
By the above discussion, this  gives a combinatorial characterization of the
infinitesimal fixed-normal rigidity of  naturally bipartite
point-hyperplane frameworks with generic hyperplane normals, i.e., the set of entries in $(a_j)_{j\in V_L}$ is algebraically independent over $\mathbb{Q}$.

\begin{theorem}
\label{thm:fixed_normal} Let $G=(V_P\cup V_L, E)$ be a  bipartite
graph. Then the following are equivalent.
\begin{enumerate}
\item The dimension of the solution space of system (\ref{eq:parallel}) is equal to $d$ for some
$(a_j)_{j\in V_L}$.
\item Every realisation of $G$ as a $d$-scene with generic hyperplane normals is trivial.
\item Every realisation of $G$ as a point-hyperplane framework in $\R^d$ with generic hyperplane normals is infinitesimally fixed-normal rigid.
\item $G$ contains a spanning subgraph $G'=(V_P\cup V_L, E')$ with $|E'|=d|V_P|+|V_L|-d$ such that
$|A|\leq d\,\nu_{V_P}(A)+\nu_{V_L}(A)-d$ for all $\emptyset \neq A\subseteq E'$.
 \item For any partition $\{A_1,\dots, A_s\}$ of  E,
 \[
 \sum_{i=1}^s (d\,\nu_{V_P}(A_i)+\nu_{V_L}(A_i)-d)  \geq  d|V_P|+|V_L|-d.
 \]
 \end{enumerate}
 \end{theorem}

\begin{proof} The equivalence of (a), (b) and (c) follows from
the above discussion.
The equivalence of (a) and (d) follows from \cite[Theorem
4.1]{WhHypergraph}.
The equivalence of (d) and (e) follows from a
result of Edmonds \cite{E} on matroids induced by submodular
functions.
\end{proof}

Note that the problem of characterising  fixed normal rigidity of
generic point-hyperplane frameworks in $\R^d$ which are not naturally bipartite is at least as difficult as that
of characterising  the rigidity of generic bar-joint frameworks in $\R^d$,
which is notoriously difficult when $d\geq 3$. We will solve the fixed normal rigidity problem when $d=2$ in the next section.

\section{Combinatorial characterizations of constrained point-line frameworks in the plane} \label{sec:constrptln}

Bar-joint frameworks with pinned vertices can be understood by deleting the corresponding columns from the rigidity matrix. In this section we investigate analogous constrained point-line frameworks with either fixed lines, fixed normals or fixed intercepts.

\subsection{Fixed-line rigidity}\label{subsec:pin}
We begin with the {\em fixed-line rigidity} of point-hyperplane frameworks in $\mathbb{R}^d$.
In this rigidity model, each line is fixed and hence has no velocity.
More formally, given a point-hyperplane framework $(G, p, \ell)$, we are interested in the following system:
\begin{align*}
\langle p_i-p_j, \dot{p}_i-\dot{p}_j\rangle&=0 && (ij\in E_{PP}) \\
\langle \dot{p}_i, a_j\rangle&=0 && (ij\in E_{PL})
\end{align*}
obtained by setting $\dot{a}_j=0$ and $\dot{r}_j=0$ in the system (\ref{eq:line_inf1_euc})-(\ref{eq:a_inf_euc}).
We say that $(G,p,\ell)$ is {\em infinitesimally fixed-line rigid} if this system has no nonzero solution.

By the results of Section~2, we know how to convert a point-hyperplane framework $(G,p,\ell)$ in $\mathbb{R}^d$
to a bar-joint framework $(G,q)$ in $\mathbb{R}^d$
in such a way that infinitesimal rigidity is preserved.\footnote{Although the transformation was presented in affine space $\mathbb{A}^d$, it can be extended to Euclidean space $\mathbb{R}^d$ by simply first lifting $p_i \ (i\in V_P)$ to $\hat{p}_i=\begin{pmatrix} p_i \\ 1\end{pmatrix}$,
applying the transformation to $(G, \hat{p}, \ell)$ to obtain $(G, \hat{q})$, and then projecting $\hat{q}_i=\begin{pmatrix} q_i \\ 1\end{pmatrix}\ (i\in V)$ to $q_i$.}
From the isomorphism between the spaces of infinitesimal motions of $(G,p,\ell)$
and $(G,q)$ (given in the proof of Theorem~\ref{thm:transfer2}), it is easy to see that $\dot{\ell}_i=0$ if and only if the corresponding $\dot{q}_i=0$ for $i\in V_L$.
This implies that  $(G, p, \ell)$ is infinitesimally fixed-line rigid if and only if
$(G, q)$ is an infinitesimally rigid bar-joint framework under the constraint that the vertices in $V_L$ are pinned.

The rigidity of {\em pinned bar-joint frameworks} is a classical concept, and in $\mathbb{R}^2$ several combinatorial characterizations are known.
Here we should be careful since, as shown in Theorem~\ref{thm:transfer2}, the points in $q(V_L)$ all lie on a line, and hence $(G,q)$ may not be a generic bar-joint framework.
Fortunately, Servatius et al.~\cite[Theorem 4]{ShaiServWh} (see also~\cite[Theorem 7.5]{KatTan})
already gave a characterization of the infinitesimal rigidity of pinned bar-joint frameworks in $\mathbb{R}^2$ in which the assumption of genericity is not required for the  positions of the pinned vertices.
This gives us the following characterization of infinitesimal fixed-line rigidity.

\begin{theorem}\label{thm:fixedline}
Let $G=(V_P\cup V_L, E)$ be a point-line graph and let $a_i\in \mathbb{S}^1$ for each $i\in V_L$.
Then $G$ can be realised as a minimally infinitesimally fixed-line rigid
point-line framework in $\mathbb{R}^2$ such that
each $i\in V_L$ is realised  as the line with normal $a_i$ if and only if
$|E|=2|V_P|$
and
\[
|F|\leq 2\nu_{V_P}(F)-3+\min\{3, 2a(F)\}
\]
for all nonempty $F\subseteq E$, where
\[
a(F):=\dim\langle a_i: i\in V_L(F)\rangle.
\]
\end{theorem}

An example illustrating Theorem~\ref{thm:fixedline} is shown in Figure~\ref{fig:pinnedlines}.

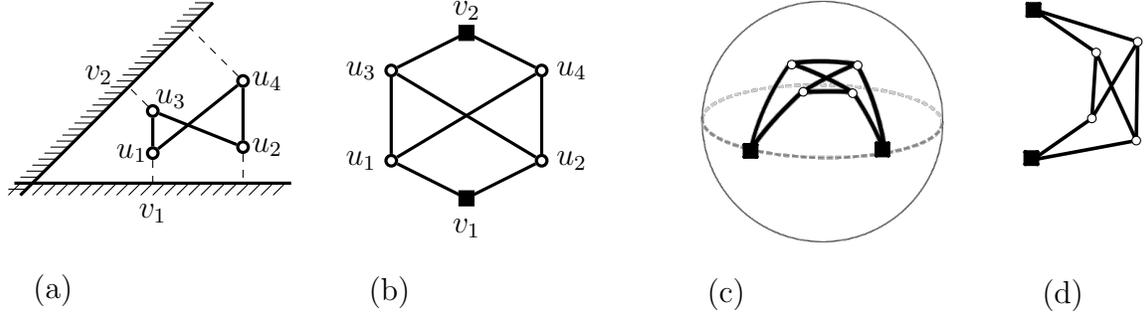
\begin{figure}[htp]
\begin{center}
\begin{minipage}{1.2cm}\begin{center}
\begin{tikzpicture}[scale=.8]
\tikzstyle{every node}=[circle, draw=black, very thick, fill=white, inner sep=0pt, minimum width=4pt];

\draw[very thick](-0.2,-0.2)--(3,3);
\draw[very thick](-0.3,0)--(4.3,0);

\draw(-0.2,0)--(-0.4,-0.2);
\draw(0,0)--(-0.2,-0.2);
\draw(0.2,0)--(0,-0.2);
\draw(0.4,0)--(0.2,-0.2);
\draw(0.6,0)--(0.4,-0.2);
\draw(0.8,0)--(0.6,-0.2);
\draw(1,0)--(0.8,-0.2);
\draw(1.2,0)--(1,-0.2);
\draw(1.4,0)--(1.2,-0.2);
\draw(1.6,0)--(1.4,-0.2);
\draw(1.8,0)--(1.6,-0.2);
\draw(2,0)--(1.8,-0.2);
\draw(2.2,0)--(2,-0.2);
\draw(2.4,0)--(2.2,-0.2);
\draw(2.6,0)--(2.4,-0.2);
\draw(2.8,0)--(2.6,-0.2);
\draw(3,0)--(2.8,-0.2);
\draw(3.2,0)--(3,-0.2);
\draw(3.4,0)--(3.2,-0.2);
\draw(3.6,0)--(3.4,-0.2);
\draw(3.8,0)--(3.6,-0.2);
\draw(4,0)--(3.8,-0.2);
\draw(4.2,0)--(4,-0.2);

\draw(0.2,0.2)--(-0.1,0.2);
\draw(0.4,0.4)--(0.1,0.4);
\draw(0.6,0.6)--(0.3,0.6);
\draw(0.8,0.8)--(0.5,0.8);
\draw(1,1)--(0.7,1);
\draw(1.2,1.2)--(0.9,1.2);
\draw(1.4,1.4)--(1.1,1.4);
\draw(1.6,1.6)--(1.3,1.6);
\draw(1.8,1.8)--(1.5,1.8);
\draw(2,2)--(1.7,2);
\draw(2.2,2.2)--(1.9,2.2);
\draw(2.4,2.4)--(2.1,2.4);
\draw(2.6,2.6)--(2.3,2.6);
\draw(2.8,2.8)--(2.5,2.8);
\draw(3,3)--(2.7,3);

\draw(-0.1,-0.1)--(-0.4,-0.1);
\draw(0.1,0.1)--(-0.2,0.1);
\draw(0.3,0.3)--(0,0.3);
\draw(0.5,0.5)--(0.2,0.5);
\draw(0.7,0.7)--(0.4,0.7);
\draw(0.9,0.9)--(0.6,0.9);
\draw(1.1,1.1)--(0.8,1.1);
\draw(1.3,1.3)--(1,1.3);
\draw(1.5,1.5)--(1.2,1.5);
\draw(1.7,1.7)--(1.4,1.7);
\draw(1.9,1.9)--(1.6,1.9);
\draw(2.1,2.1)--(1.8,2.1);
\draw(2.3,2.3)--(2,2.3);
\draw(2.5,2.5)--(2.2,2.5);
\draw(2.7,2.7)--(2.4,2.7);
\draw(2.9,2.9)--(2.6,2.9);

\node [draw=white, fill=white,rectangle] (a) at (1.1,1.8) {$v_2$};
\node [draw=white, fill=white,rectangle] (a) at (2,-0.5) {$v_1$};
\node [draw=white, fill=white,rectangle] (a) at (1.65,0.5) {$u_1$};
\node [draw=white, fill=white,rectangle] (a) at (3.9,0.6) {$u_2$};
\node [draw=white, fill=white,rectangle] (a) at (2.3,1.4) {$u_3$};
\node [draw=white, fill=white,rectangle] (a) at (3.9,1.7) {$u_4$};

\path (2,0.5) node (u1) {};
\path (3.5,0.6) node (u2) {};
\path (2,1.2) node (u3) {};
\path (3.5,1.7) node (u4) {};

\draw[dashed](u1)--(2,0);
\draw[dashed](u2)--(3.5,0);

\draw[very thick](u1)--(u3);
\draw[very thick](u2)--(u4);
\draw[very thick](u1)--(u4);
\draw[very thick](u2)--(u3);

\draw[dashed](u3)--(1.6,1.6);
\draw[dashed](u4)--(2.6,2.6);


\end{tikzpicture}
\\\vspace{0.48cm} (a) \end{center}\end{minipage}
\hspace{3cm}
\begin{minipage}{1.2cm}\begin{center}
\begin{tikzpicture}[scale=1]
\tikzstyle{every node}=[circle, draw=black, very thick, fill=white, inner sep=0pt, minimum width=4pt];

\path (0,0.5) node (u1) {};
\path (2,0.5) node (u2) {};
\path (0,1.7) node (u3) {};
\path (2,1.7) node (u4) {};

\filldraw(0.9,2.1) rectangle (1.1,2.3);
\filldraw(0.9,-0.1) rectangle (1.1,0.1);


\node [draw=white, fill=white,rectangle] (a) at (1.,2.5) {$v_2$};
\node [draw=white, fill=white,rectangle] (a) at (1,-0.4) {$v_1$};
\node [draw=white, fill=white,rectangle] (a) at (-0.4,0.5) {$u_1$};
\node [draw=white, fill=white,rectangle] (a) at (2.4,0.5) {$u_2$};
\node [draw=white, fill=white,rectangle] (a) at (-0.4,1.7) {$u_3$};
\node [draw=white, fill=white,rectangle] (a) at (2.4,1.7) {$u_4$};

\draw[very thick](u1)--(u3);
\draw[very thick](u2)--(u4);
\draw[very thick](u1)--(u4);
\draw[very thick](u2)--(u3);

\draw[very thick](1,0)--(u1);
\draw[very thick](1,0)--(u2);
\draw[very thick](1,2.2)--(u3);
\draw[very thick](1,2.2)--(u4);


\end{tikzpicture}
\\\vspace{0.3cm} (b) \end{center}\end{minipage}
\hspace{3cm}
\begin{minipage}{1.2cm}\begin{center}{\includegraphics[scale=.27]{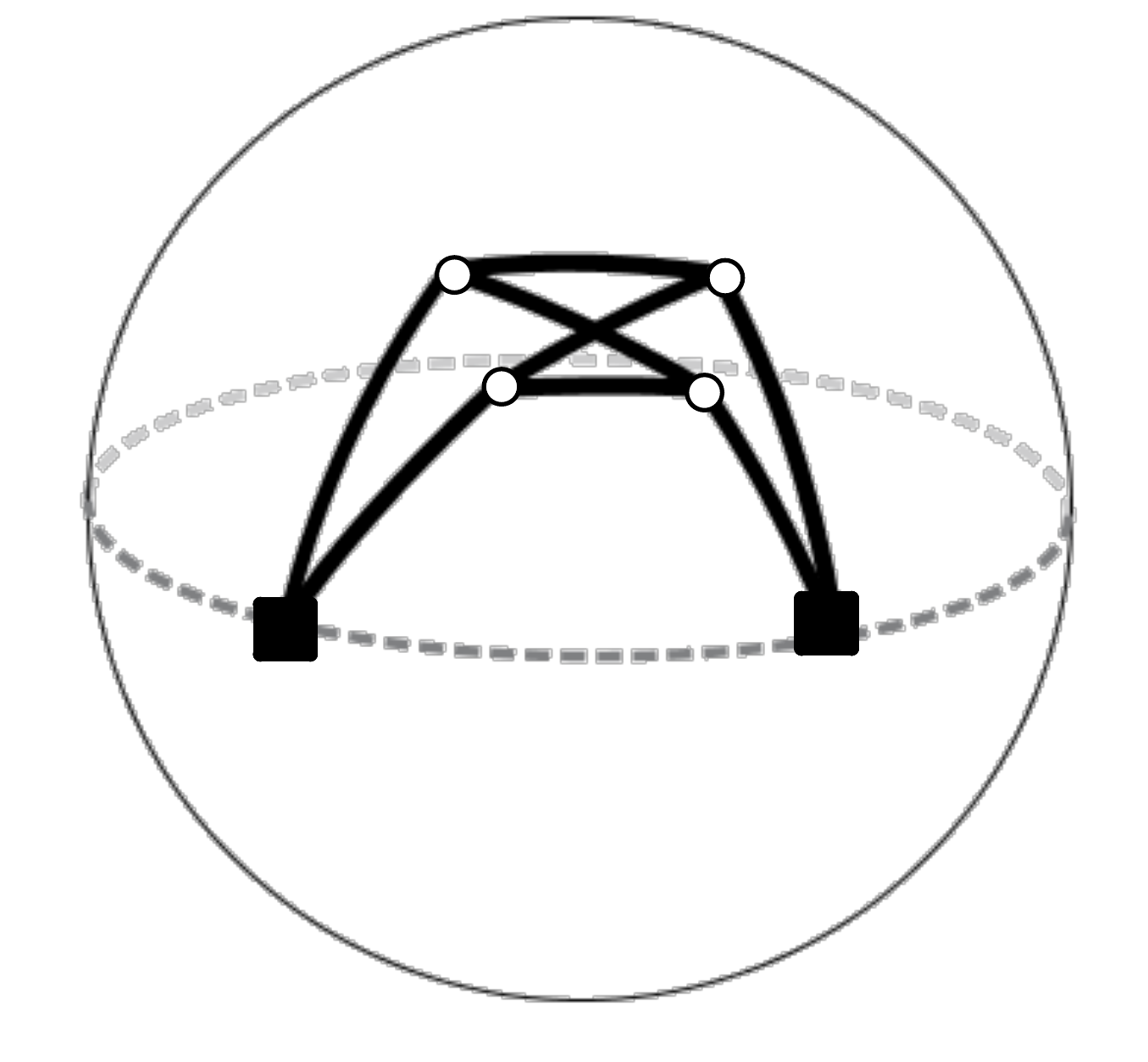}}\\\vspace{0.1cm} (c) \end{center}\end{minipage}
\hspace{3cm}
\begin{minipage}{1.2cm}\begin{center}{\includegraphics[scale=.27]{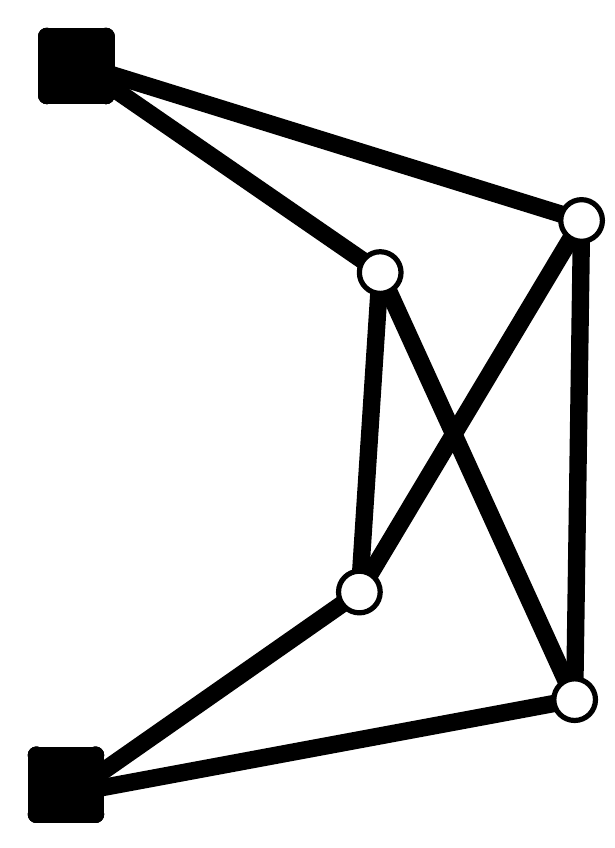}}\\\vspace{1.25cm} (d) \end{center}\end{minipage}
\caption{A minimally infinitesimally fixed-line rigid point-line framework (a) and its associated point-line graph (b).
The corresponding framework on the sphere (c) and its projection to the plane (d).
}
\label{fig:pinnedlines}
\end{center}
\end{figure}

\subsection{Fixed-normal rigidity}\label{subsec:normal}

We introduced the fixed-normal rigidity of a point-hyperplane framework $(G,p,\ell)$ in Section \ref{sec:scene} and observed that the infinitesimal motions of $(G,p,\ell)$ which preserve the normals of the hyperplanes are determined by the system of equations
\begin{align}
\langle p_i - p_j, \dot{p}_i-\dot{p}_j\rangle&=0 && (ij\in E_{PP}) \label{eq:line_inf1fixednormal} \\
\langle \dot{p}_i, a_j\rangle+ \dot r_j&=0 && (ij\in E_{PL})\,. \label{eq:line_inf2fixednormal}
\end{align}
Owen and Power \cite{OP} had previously used a recursive construction to characterise the fixed-normal rigidity of generic point-line frameworks in $\mathbb{R}^2$. We will show that their result can be deduced from Theorem \ref{thm:barjointcol}.

\begin{thm}\label{thm:fixednormal}
Let $G=(V_P\cup V_L,E)$ be a point-line graph with $|V_P|\geq1$ and $|V_L|\geq2$ and $T$ be the edge set of a tree with vertex set $V_L$. Then the following statements are equivalent:\\
(a) $G$ can be realised as a  point-line  framework in $\mathbb{R}^2$ which is minimally infinitesimally fixed-normal rigid;\\
(b) $G+T$ can be realized as an infinitesimally rigid bar-joint framework in $\mathbb{R}^2$ such that the points assigned to $V_L$ are collinear;\\
(c) $|E|= 2|V_P|+|V_L|-2$,
$|F|\leq 2\nu_{V_P}(F)-3$ for all $\emptyset\neq F\subseteq E$ with $\nu_{V_L}(F)=0$, and
$|F|\leq 2\nu_{V_P}(F)+\nu_{V_L}(F)-2$ for all $\emptyset\neq F\subseteq E$.
\end{thm}
\begin{proof}
It is straightforward to show that (a) implies (c).

Suppose that $G$ satisfies (c).  We will show that $G+T$ satisfies (b) by showing it satisfies
the conditions of Theorem \ref{thm:barjointcol}(c) with $V\sm X=V_P$ and $X=V_L$. Since $|E|= 2|V_P|+|V_L|-2$, $G+ T$ has $2|V|-3$ edges.
Choose a nonempty $A\subseteq E\cup T$, let $\scra=\{A_1,A_2,\ldots,A_s\}$ be a partition of $A$ and let $\scra'=\{A_i\in \scra\,:\,A_i\sm T\neq \emptyset\}$. Then  \begin{eqnarray*}
\sum_{A_i\in \scra}(2\nu_{V_P}(A_i)+\nu_{V_L}(A_i)-2)+\nu_{V_L}(A)-1 &\geq& \sum_{A_i\in \scra'}(2\nu_{V_P}(A_i\sm T)+\nu_{V_L}(A_i\sm T)-2)+|A\cap T|\\
&\geq& \sum_{A_i\in \scra'}|A_i\sm T|+|A\cap T|=|A|.
\end{eqnarray*}
Thus $G+ T$ satisfies the condition of Theorem \ref{thm:barjointcol}(c). Hence $G+T$ also satisfies Theorem \ref{thm:barjointcol}(a) so (b) holds

Finally we suppose that (b) holds. Then
$G+ T$ can be realised as an infinitesimally rigid point-line framework $(G+ T,p,\ell)$.
This implies that the dimension of the solution space of the system (\ref{eq:line_inf1_euc})-(\ref{eq:a_inf_euc}) for $(G+ T,p,\ell)$ is equal to three.
Choose a special vertex  $i^*\in V_L$, and add the extra constraint
\begin{equation}
\label{eq:4_2_1}
\langle a_{i^*}^{\bot}, \dot{a}_{i^*}\rangle=0
\end{equation}
to the system (\ref{eq:line_inf1_euc})-(\ref{eq:a_inf_euc}), where $x^{\bot}$ denotes the $\pi/2$ clockwise rotation of a vector $x\in \mathbb{R}^2$.
 Since  the system (\ref{eq:line_inf1_euc})-(\ref{eq:a_inf_euc}) contains a rotation in its solution space, adding the extra equation (\ref{eq:4_2_1}) decreases the dimension of the solution space by one.

 Note that, in  the system (\ref{eq:line_inf1_euc})-(\ref{eq:a_inf_euc}) for $(G+ T, p,\ell)$,
 each edge in $T$ gives the following constraint:
 \begin{equation}
\label{eq:4_2_2}
\langle a_i, \dot{a}_j\rangle+\langle \dot{a}_i, a_j\rangle=0 \qquad (ij\in T).
\end{equation}
A simple inductive argument, starting from $i^*$, implies that (\ref{eq:4_2_1}), (\ref{eq:4_2_2}), and (\ref{eq:a_inf_euc}) hold
if and only if
\begin{equation}
\label{eq:4_2_3}
\dot{a}_j=0\qquad (j\in V_L).
\end{equation}
Since the combination of (\ref{eq:line_inf1})-(\ref{eq:a_inf}) with (\ref{eq:4_2_3}) is equivalent to the system (\ref{eq:line_inf1fixednormal})-(\ref{eq:line_inf2fixednormal}) for $(G,p,\ell)$, we conclude that the dimension of  the solution space of the  latter system is   equal to two. In other words, $(G,p,\ell)$ admits only trivial infinitesimal motions as a fixed-normal point-line framework and (a) holds.
\end{proof}

An example illustrating Theorem \ref{thm:fixednormal} is shown in Figure \ref{fig:fixednormal}.

\begin{figure}[htp]
\begin{center}
\begin{minipage}{1.2cm}\begin{center}
\begin{tikzpicture}[scale=0.7]
\tikzstyle{every node}=[circle, draw=black, very thick, fill=white, inner sep=0pt, minimum width=4pt];

\node [draw=white, fill=white,rectangle] (a) at (0.5,2.4) {$u_3$};
\node [draw=white, fill=white,rectangle] (a) at (0.05,1) {$u_2$};
\node [draw=white, fill=white,rectangle] (a) at (-0.8,2) {$u_1$};

\node [draw=white, fill=white,rectangle] (a) at (0,3.4) {$v_3$};
\node [draw=white, fill=white,rectangle] (a) at (1.8,1.2) {$v_2$};
\node [draw=white, fill=white,rectangle] (a) at (-1.8,1.2) {$v_1$};

\draw[very thick](0.3,-0.3)--(-3.3,3.3);
\draw[very thick](-0.3,-0.3)--(3.3,3.3);
\draw[very thick](-3.5,3.1)--(3.5,3.1);

\path (-0.6,1.7) node (u1) {};
\path (0.4,1.2) node (u2) {};
\path (0.1,2.2) node (u3) {};

\draw[very thick](u2)--(u3);
\draw[very thick](u1)--(u3);
\draw[very thick](u2)--(u1);

\draw[dashed](u3)--(0.1,3.3);
\draw[dashed](u2)--(0.8,0.8);
\draw[dashed](u1)--(-1.15,1.15);

\draw[very thick,gray,->](u3)--(-0.4,2.2);
\draw[very thick,gray,->](u2)--(0.65,1.45);
\draw[very thick,gray,->](u1)--(-0.6,1.1);
\draw[very thick,gray,->](-2.5,2.5)--(-3,2);
\draw[very thick,gray,->](-0.5,0.5)--(-1,0);


\end{tikzpicture}
\\\vspace{0.6cm} (a) \end{center}\end{minipage}
\hspace{4cm}
\begin{minipage}{1.2cm}\begin{center}
\begin{tikzpicture}[scale=.9]
\tikzstyle{every node}=[circle, draw=black, very thick, fill=white, inner sep=0pt, minimum width=4pt];

\node [draw=white, fill=white,rectangle] (a) at (0,0.55) {$u_3$};
\node [draw=white, fill=white,rectangle] (a) at (1.2,-0.1) {$u_2$};
\node [draw=white, fill=white,rectangle] (a) at (-1.17,-0.1) {$u_1$};

\node [draw=white, fill=white,rectangle] (a) at (0,2.86) {$v_3$};
\node [draw=white, fill=white,rectangle] (a) at (1.2,2.15) {$v_2$};
\node [draw=white, fill=white,rectangle] (a) at (-1.17,2.15) {$v_1$};

\path (-0.8,0) node (u1) {};
\path (0.8,0) node (u2) {};
\path (0,1) node (u3) {};


\filldraw[draw=black](-0.8,1.85)-- (-0.65,2)-- (-0.8,2.15)-- (-0.95,2)--(-0.8,1.85);
\filldraw[draw=black](0.8,1.85)-- (0.95,2)-- (0.8,2.15)-- (0.65,2)--(0.8,1.85);
\filldraw[draw=black](0,2.35)-- (0.15,2.5)-- (0,2.65)-- (-0.15,2.5)--(0,2.35);

\draw[very thick](u2)--(u3);
\draw[very thick](u1)--(u3);
\draw[very thick](u2)--(u1);

\draw[very thick](u1)--(-0.8,2);
\draw[very thick](u2)--(0.8,2);
\draw[very thick](u3)--(0,2.5);


\end{tikzpicture}
\\\vspace{0.45cm} (b) \end{center}\end{minipage}
\hspace{2cm}
\begin{minipage}{1.2cm}\begin{center}
{\includegraphics[scale=.28]{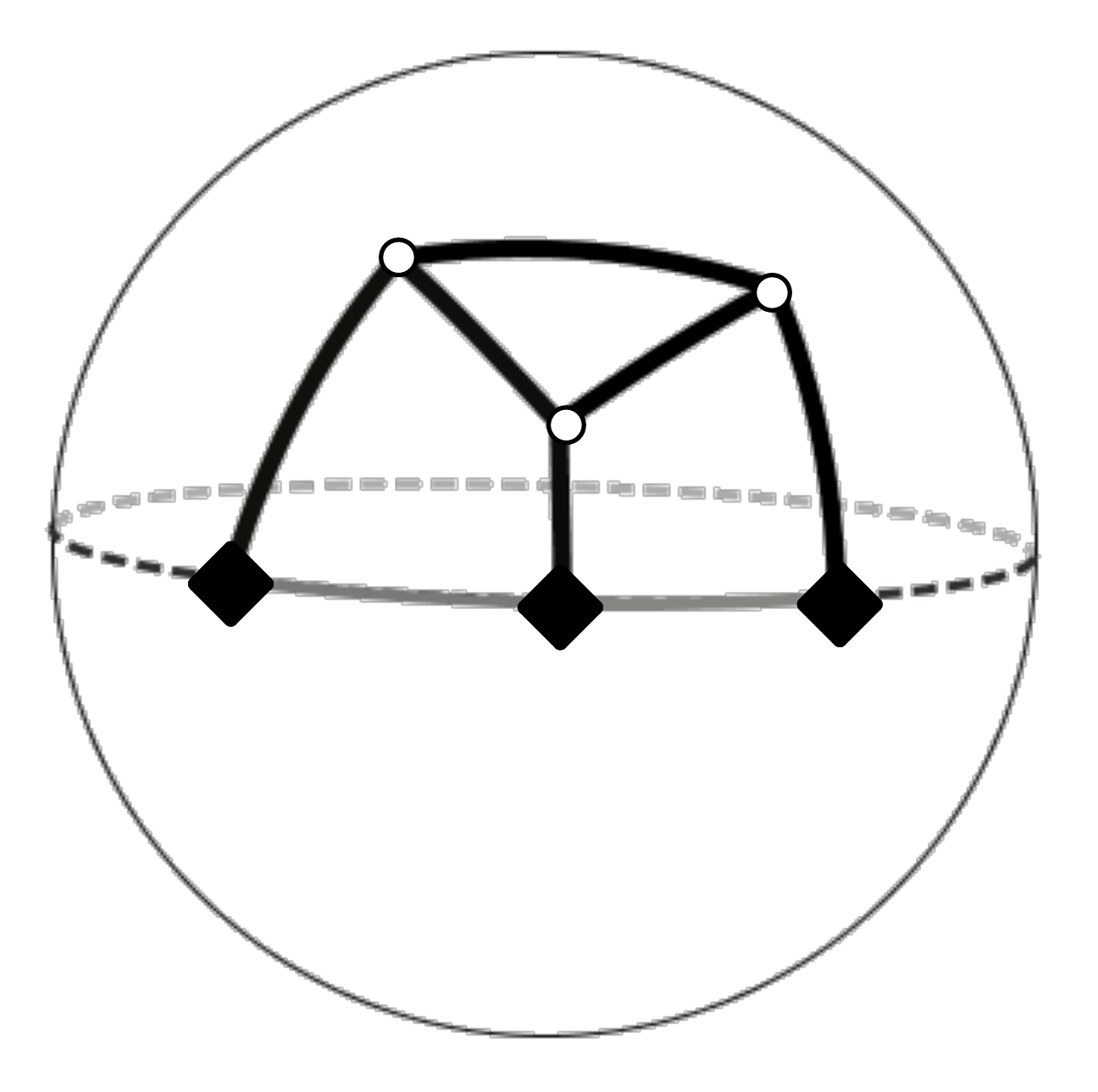}}
\\\vspace{-0.3cm} (c) \end{center}\end{minipage}
\hspace{3cm}
\begin{minipage}{1.2cm}\begin{center}
{\includegraphics[scale=.43]{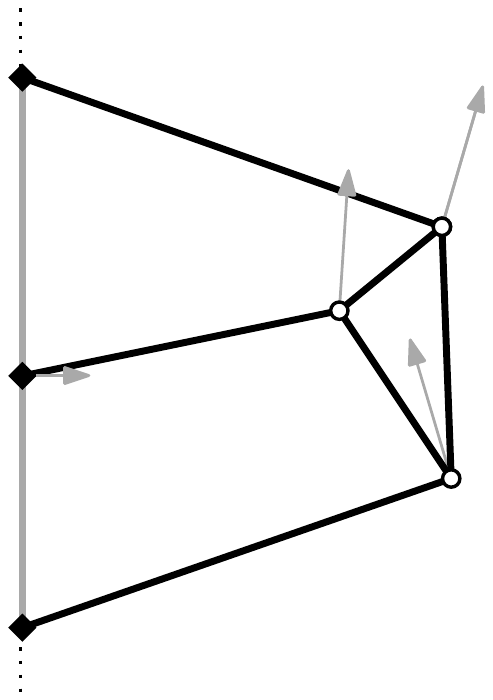}}
\\\vspace{0.3cm} (d) \end{center}\end{minipage}
\caption{An infinitesimally flexible fixed-normal  point-line framework (a) and   its associated point-line graph (b).
The corresponding framework on the sphere (c) and its projection to the plane (d) in which the line-vertics are collinear and constraints that the lines have fixed normals are modeled by adding a tree of grey edges between the three line-vertices. The point-line framework is fixed-normal flexible since it has only six constraints and $2|V_P|+|V_L|-2=7$.}
\label{fig:fixednormal}
\end{center}
\end{figure}

\subsection{Fixed-intercept rigidity}\label{subsec:intercept}

We now consider point-line frameworks in which each line is allowed to rotate about some fixed point but cannot translate. Such a framework will have at most one trivial motion, and this will exist only when each of the lines are allowed to rotate about the same point. We will focus on the special case when all of the lines are concurrent and are allowed to rotate about their common point of intersection.
We will refer to such a point-line framework as a {\em line-concurrent framework}.

Given a line-concurrent  framework $(G,p,\ell)$, we may always assume that the common intersection point of the lines is the origin, i.e., $r_j=0$ for all $j\in V_L$, and hence the fixed-intercept constraint implies that $\dot{r}_j=0$ for all $j\in V_L$. Substituting $\dot{r}_j=0$ into (\ref{eq:line_inf1_euc})-(\ref{eq:a_inf_euc}), we deduce that the infinitesimal motions are determined by the following system:
\begin{align}
\langle p_i - p_j, \dot{p}_i-\dot{p}_j\rangle&=0 && (ij\in E_{PP}) \label{eq:intercept_sys1} \\
\langle p_i, \dot{a}_j\rangle+\langle \dot{p}_i, a_j\rangle&=0 && (ij\in E_{PL}) \\
\langle a_i, \dot{a}_j\rangle+\langle \dot{a}_i, a_j\rangle&=0 && (ij\in E_{LL}) \\
\langle a_i, \dot{a}_i\rangle&=0 && (i\in V_L). \label{eq:intercept_sys2}
\end{align}
We say that $(G,p,\ell)$ is {\em infinitesimally fixed-intercept rigid} if the above system admits only the trivial infinitesimal motion.

Our theorem gives a characterization of infinitesimal fixed-intercept rigidity even in the case when the normals of the lines are specified as input without assuming genericity.
(We will see below that allowing arbitrary normals gives potential applications to engineering.)

\begin{theorem}
\label{thm:fixed_intercept}
Let $G=(V_P\cup V_L, E)$ be a point-line graph with $|V_L|\geq 2$ and
let $a_i\in \mathbb{S}^1$ for each $i\in V_L$.
Suppose that each line has a distinct normal.
Then $G$ can be realised as a minimally infinitesimally fixed-intercept rigid line-concurrent framework such that each
$i\in V_L$ is realised as the line with normal $a_i$ if and only if $|E|=2|V_P|+|V_L|-1$
and
\[
|F|\leq 2\nu_{V_P}(F)+\nu_{V_L}(F)-3+\min\{2,|V_L(F)|\}
\]
for all nonempty $F\subseteq E$.
\end{theorem}

We will in fact prove a  stronger statement, in which lines are allowed to have the same normal (as in the setting of Theorem~\ref{thm:fixedline}).
To state the result we need the following notation.
For a point-line graph $G=(V_P\cup V_L, E)$,
let $G^P$ be the graph on $V_P$ obtained from $G$ by
removing $V_L$ and regarding each edge $ij$ in $E_{PL}$ with $i\in V_P$ as a loop at $i$.
Similarly, let $G^L$ be the graph on $V_L$ obtained from $G$ by
removing $V_P$ and regarding each edge $ij$ in $E_{PL}$ with $j\in V_L$ as a loop at $j$.
For an edge set $F$ of $G$, let $G[F]$ be the subgraph of $G$ induced by $F$.
Also for a graph $H$, let $C(H)$ be the set of connected components in $H$.

\begin{theorem}
\label{thm:fixed_intercept0}
Let $G=(V_P\cup V_L, E)$ be a point-line graph with $|V_L|\geq 2$ and
let $a_i\in \mathbb{S}^1$ for each $i\in V_L$.
Then $G$ can be realised as a minimally infinitesimally fixed-intercept rigid line-concurrent framework such that each
$i\in V_L$ is realised as the line with normal $a_i$ if and only if
\begin{itemize}
\item $|E|=2|V_P|+|V_L|-1$,
\item $a_i\neq a_j$ for each $ij\in E_{LL}$, and
\item $|F|\leq 2\nu_{V_P}(F)+\nu_{V_L}(F)-1-
\sum_{H\in C((G[F])^P)}(2-\dim \langle a_j: ij\in F\cap E_{PL}, i\in V(H)\rangle)$
for all nonempty $F\subseteq E$.
\end{itemize}
\end{theorem}

Consider the point-line graph $G$ shown in Figure~\ref{fig:fixed_intercept}(a). Two different realisations as a line-concurrent point-line framework are shown in (d) and (e). The framework in (d) has two lines with the same normal. We can use Theorem~\ref{thm:fixed_intercept0} to show that it is not infinitesimally fixed intercept rigid by taking $F$ to be the edge-set of sugraph of $G$ shown in (b). Then $G[F]^P$ is as shown in (c) and the right hand side of the inequality of Theorem~\ref{thm:fixed_intercept0} is $2\cdot 4+2-1-2=7$, which is greater than $|F|=8$. On the other hand the realisaton shown in (e) is infinitesimally fixed intercept rigid. In particular if we evaluate the right hand side of the inequality of Theorem~\ref{thm:fixed_intercept} for $F$, we obtain $2\cdot 4+2-1=9$ so the inequality holds.

\begin{figure}[htp]
\centering
\begin{minipage}{0.2\textwidth}
\centering
\includegraphics[scale=1]{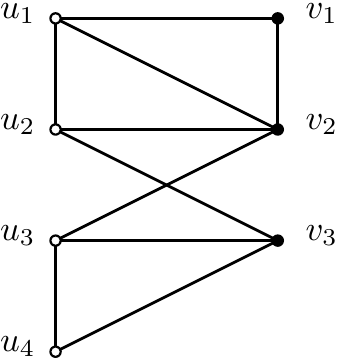}
\par
(a)
\vspace{.7cm}
\end{minipage}
\hspace{1.4cm}
\begin{minipage}{0.2\textwidth}
\centering
\includegraphics[scale=1]{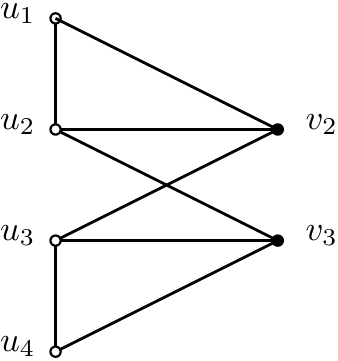}
\par
(b)
\vspace{.7cm}
\end{minipage}
\hspace{1cm}
\begin{minipage}{0.2\textwidth}
\centering
\includegraphics[scale=1]{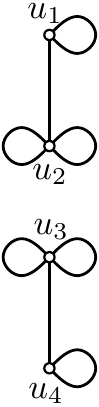}
\par
(c)
\vspace{.7cm}
\end{minipage}
\vspace{.7cm}
\begin{minipage}{0.4\textwidth}
\centering
\includegraphics[scale=0.9]{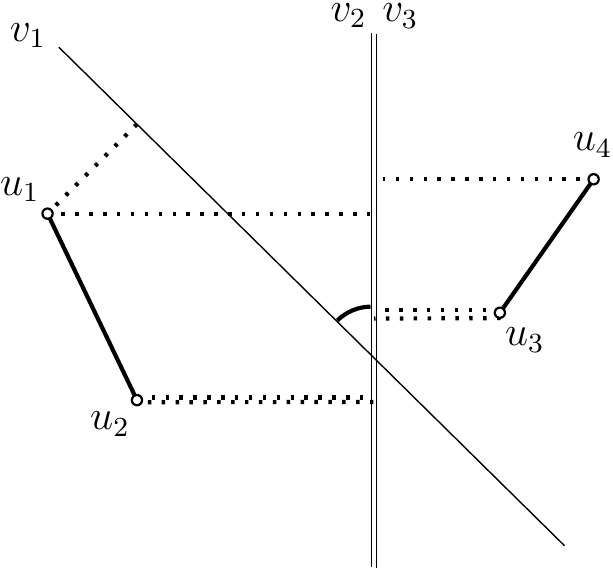}
\par
(d)
\end{minipage}
\begin{minipage}{0.4\textwidth}
\centering
\includegraphics[scale=0.9]{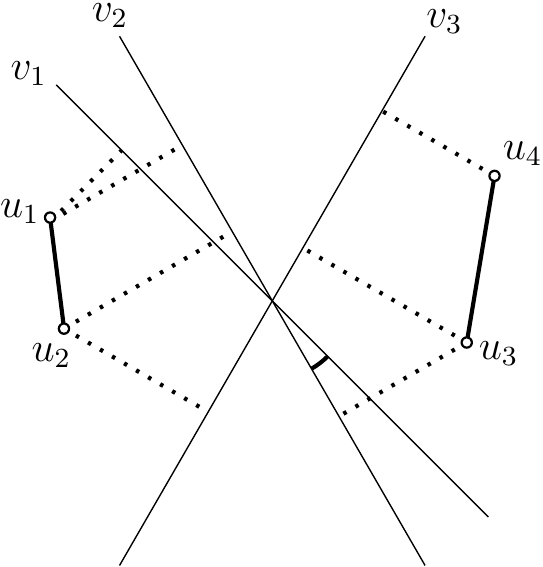}
\par
(e)
\end{minipage}

\begin{minipage}{0.4\textwidth}
\centering
\includegraphics[scale=0.7]{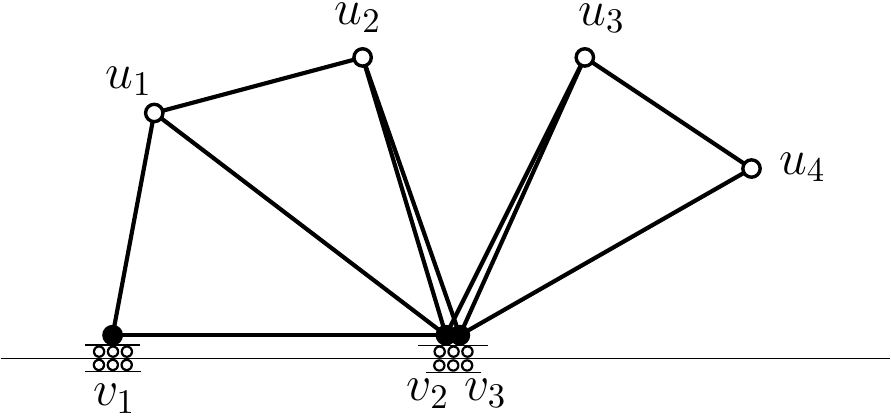}
\par
(f)
\end{minipage}
\begin{minipage}{0.4\textwidth}
\centering
\includegraphics[scale=0.7]{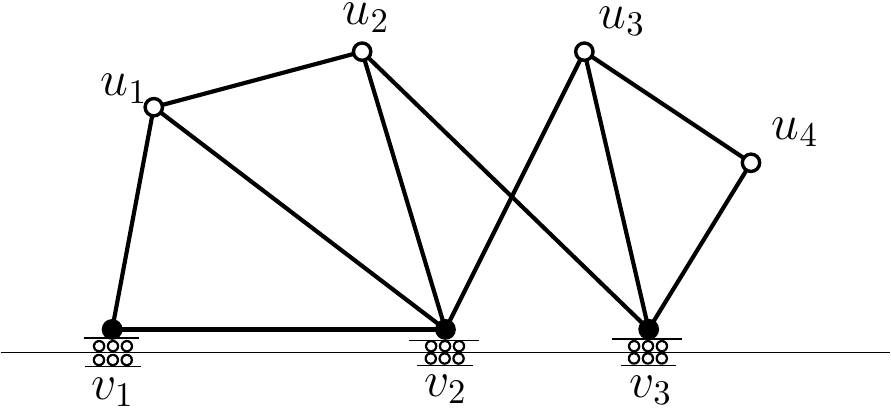}
\par
(g)
\end{minipage}
\caption{(a) $G=(V_P\cup V_L, E)$, where the left side  is $V_P$ and the right side is $V_L$.
(b) an edge set $F$ violating the count in Theorem~\ref{thm:fixed_intercept0} when  two normals  coincide as in (d). (c) $G[F]^P$.
(d) A line-concurrent realization such that the two lines $u_1, u_2$ have the same normal.
(e) A generic line-concurrent realization.
(f)(g) Bar-joint frameworks with horizontal slider joints corresponding to (d)(e).}
\label{fig:fixed_intercept}
\end{figure}

We will see in the next section that the generic version of  Theorem~\ref{thm:fixed_intercept} can be deduced from Theorem~\ref{thm:barjointcol}. However there seems to be no such reduction in the nongeneric case so we provide a direct proof.
We  first give several tools from matroid theory in Section~\ref{subsec:matroid} and
delay the proof of Theorem~\ref{thm:fixed_intercept0} to Section~\ref{subsec:proof1}.
Finally in Section~\ref{subsec:proof2} we show that, if each line has a distinct normal,
 the combinatorial condition of Theorem~\ref{thm:fixed_intercept0} is equivalent to that of Theorem~\ref{thm:fixed_intercept}.
 Although the combinatorial condition in Theorem~\ref{thm:fixed_intercept} is much simpler than that in Theorem~\ref{thm:fixed_intercept0},
 currently we have no direct proof of Theorem~\ref{thm:fixed_intercept}.

Before proceeding to the proof, we describe a consequence of Theorem~\ref{thm:fixed_intercept0} for bar-joint frameworks.
Consider, again, the transformation given in  Section~2, which converts a point-line framework $(G,p,\ell)$ to a bar-joint framework $(G,q)$.
Note that a line-concurrent point-line framework $(G,p,\ell)$ will be mapped to a bar-joint framework $(G,q)$ such that all the  points in $q(V_L)$ lie on a line, say a horizontal line.
If the rotation on the sphere is done such that the north pole is mapped to a point on the equator (so that the north pole is finally mapped to a point at infinity after the projection to the plane), then in the isomorphism between the spaces of infinitesimal motions of $(G,p,\ell)$ and $(G,q)$, we have that
$\dot{r}_j=0$ if and only if $\dot{q}_j$  is in the horizontal direction.
In other words each point $q(v)$ for $v\in V_L$ can only slide along the horizontal line.
Therefore,  the question about the fixed-intercept rigidity of $(G,p,\ell)$ can be rephrased as the rigidity question of bar-joint frameworks with horizontal slider joints on the ground. This transformation is illustrated in  
Figure~\ref{fig:fixed_intercept}(d),(e),(f),(g) and Figure~\ref{fig:sphereintercept}(a),(b).

More formally, a bar-joint framework with horizontal slider joints is a tuple $(G,X,p)$ of a graph $G$, $X\subseteq V(G)$, and $p:V\rightarrow \mathbb{R}^2$,
where $X$ will represent a set of slider joints.
An infinitesimal motion $\dot{p}$ of $(G,X,p)$ is an infinitesimal motion of $(G,p)$ with $\dot{p}(v) \cdot \begin{pmatrix} 0 \\ 1\end{pmatrix}=0$ for all $v\in X$,
and $(G, X, p)$ is said to be infinitesimally rigid if horizontal translations are the only possible infinitesimal motions of $(G, X, p)$.
By the rigidity transformation explained above, Theorem~\ref{thm:fixed_intercept0} can be restated  as follows.
\begin{theorem}
\label{thm:horizontal}
Let $G=(V_P\cup V_L, E)$ be a point-line graph with $|V_L|\geq 2$ and
let $x_i\in \mathbb{R}^1$ for each $i\in V_L$.
Then $G$ can be realised as a minimally infinitesimally rigid bar-joint framework  in $\mathbb{R}^2$ with $V_L$ as a set of horizontal slider joints such that the coordinate of $i\in V_L$ is $\begin{pmatrix} x_i \\ 0\end{pmatrix}$  if and only if
\begin{itemize}
\item $|E|=2|V_P|+|V_L|-1$,
\item $x_i\neq x_j$ for each $ij\in E_{LL}$, and
\item
$|F|\leq 2\nu_{V_P}(F)+\nu_{V_L}(F)-1-
\sum_{H\in C((G[F])^P)}\max\{0, 2-|\{x_j : ij\in F\cap E_{PL}, i\in V(H)\}|\} )$
for all nonempty $F\subseteq E$.
\end{itemize}
\end{theorem}

Note that Theorem~\ref{thm:horizontal} has no restriction on the coordinates of the slider joints.
This is a much stronger statement than previous results~\cite{StreinuTheran,KatTan},  where a certain genericity is assumed for the coordinates of slider joints.
Such bar-joint frameworks with horizontal sliders frequently appear in the structural engineering literature, where sliders are often located on the horizontal ground.


\begin{figure}[h]
\centering
\begin{minipage}{0.4\textwidth}
\centering
\includegraphics[scale=0.37]{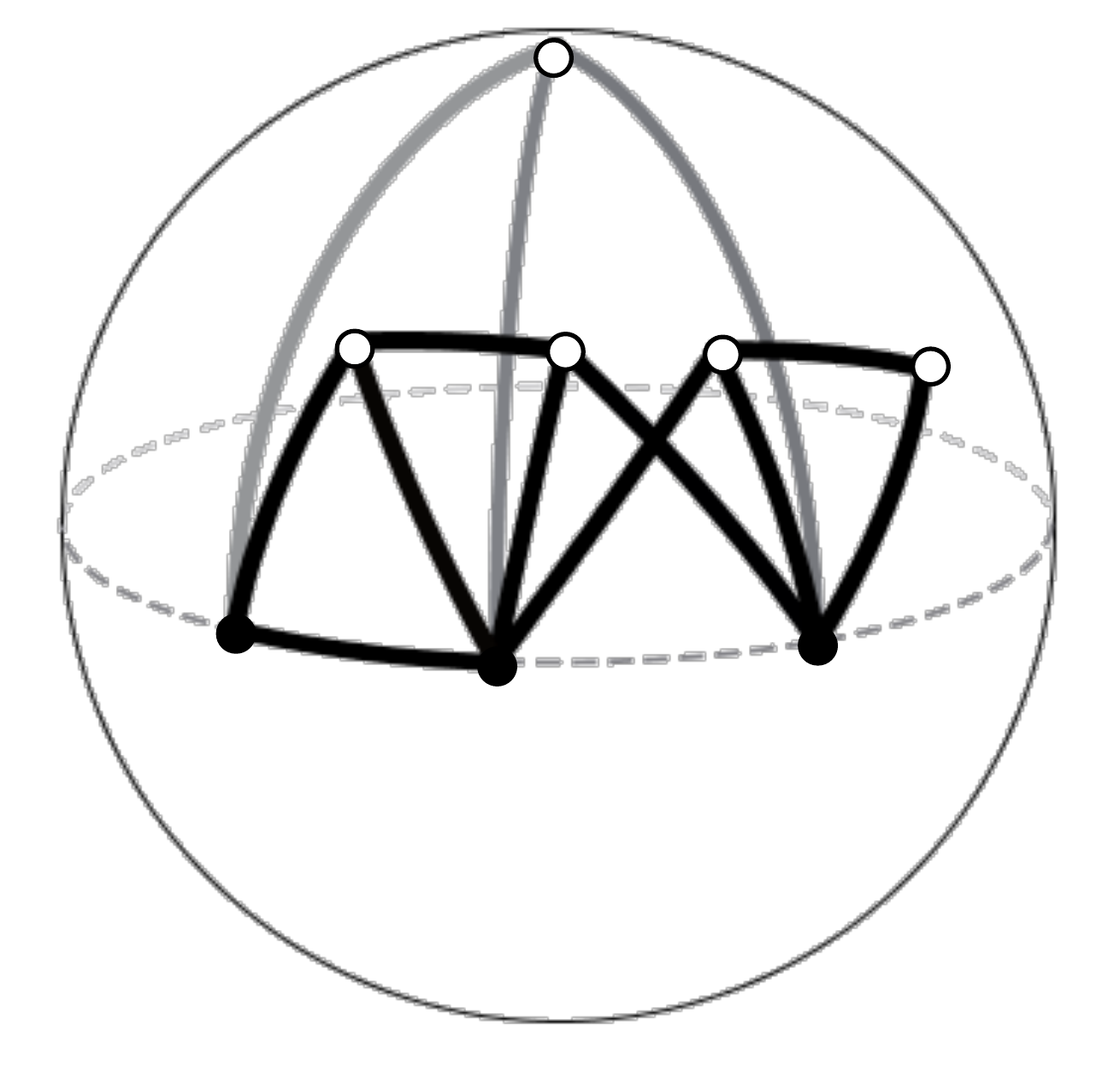}\quad\quad
\par
(a)
\end{minipage}
\begin{minipage}{0.4\textwidth}
\centering
\includegraphics[scale=0.37]{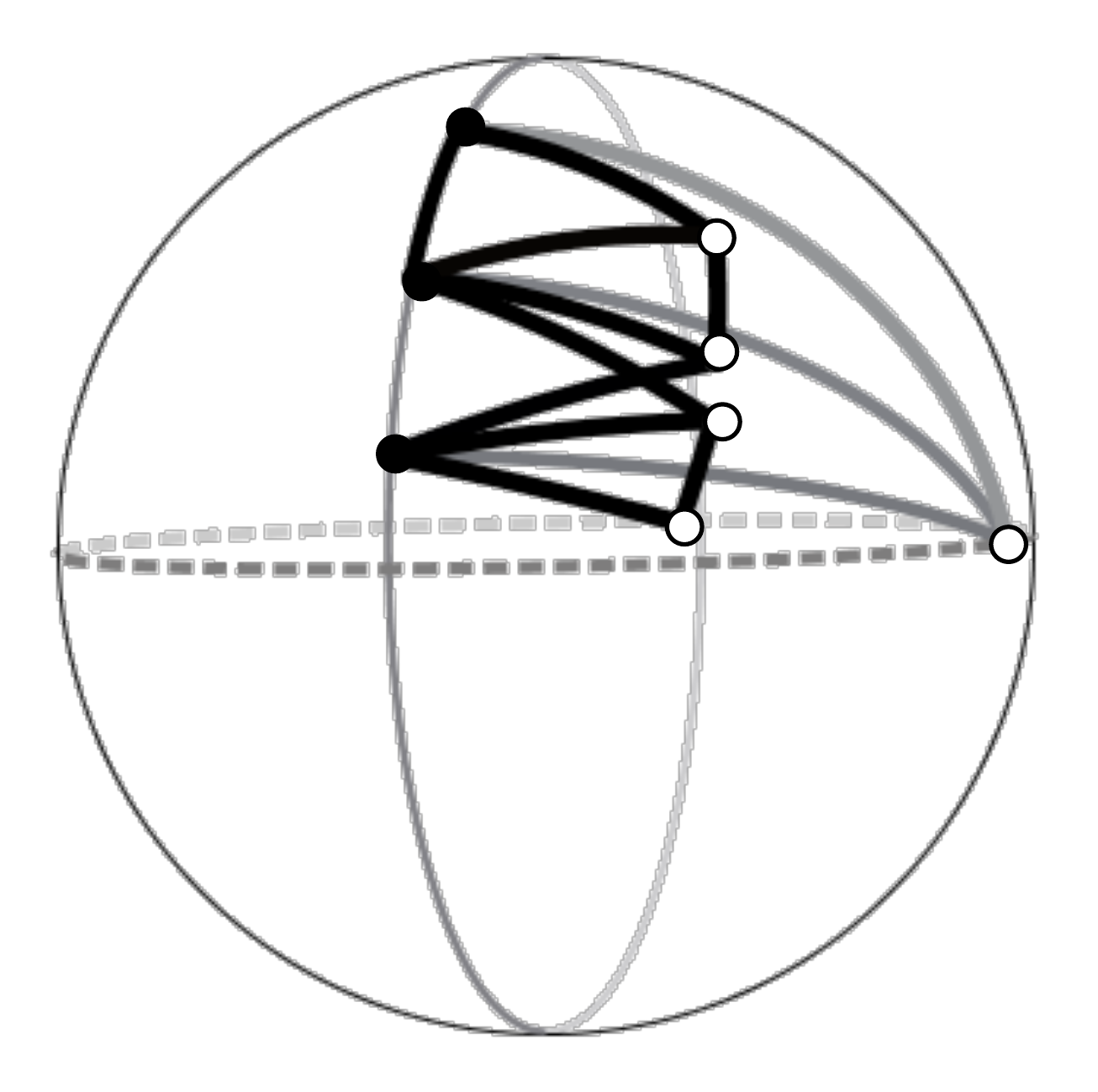}\quad\quad
\par
(b)
\end{minipage})
\caption{(a) shows the spherical framework  corresponding to the point-line framework in Figure 7(e), in which we have added a new point-vertex joined to the three line-vertices by grey edges to model the constraints that the lines must rotate about a fixed point.   We rotate this framework to move the line-vertices off the equator and the new point-vertex onto the equator to obtain the spherical framework in (b). This projects to the bar-joint framework in Figure 7(g). The three line-vertices project to three collinear joints and the new point-vertex and its incident grey edges correspond to the constraints that the collinear joints are forced to move along the line.
}
\label{fig:sphereintercept}
\end{figure}

\subsubsection{Matroid preliminaries}
\label{subsec:matroid}
Let $G=(V,E)$ be a graph which may contain loops and let $d$ be a positive integer.
We assign a copy of $\mathbb{R}^d$ to each vertex and let $(\mathbb{R}^d)^V$ be the direct sum of those spaces over all vertices. For $x\in (\mathbb{R}^d)^V$, let $x(i)\in \mathbb{R}^d$ be the restriction of $x$ to the space assigned to $i\in V$.
Consider the incidence matrix of an  oriented $G$, that is, the $(|E| \times |V|)$-matrix in which the entries in row $e=ij$ with $i<j$ are $1$ in column $i$,
 $-1$ in column $j$ and $0$ elsewhere,
 and the entries in a row corresponding to a loop at $i$ are $1$ in column $i$ and $0$ elsewhere.
 This matrix gives a linear representation of a variant of the {\em cycle matroid} of $G$.
This matroid has rank equal to
 \[|V|-\sum_{H\in C(G)} \lambda(H),\]
 where $\lambda(H):=1$ if $H$ has no loop and  $\lambda(H)=0$ otherwise.
 We can obtain a linear representation of this matroid by assigning a one-dimensional vector space
 $A_e$ to each $e\in E$, where
 \begin{align*}
 A_e&=
 \{x\in \mathbb{R}^V: x(i)+x(j)=0, x(k)=0\ \forall k\in V\setminus \{i,j\}\} && (\text{$e=ij$ is not a loop}) \\
 A_e&=\{x\in \mathbb{R}^V: x(k)=0\ \forall k\in V\setminus \{i\}\} && (\text{$e$ is a loop at $i$})
 \end{align*}
 with each edge $e\in E$.
 Then $\dim \langle A_e: e\in E\rangle=|V|-\sum_{H\in C(G)} \lambda(H)$.

 Next we take the direct sum of $d$ copies of $A_e$, which gives a $d$-dimensional vector space $A_e^d$ for each edge $e$:
 \begin{align*}
 A_e^d&=
 \{x\in (\mathbb{R}^d)^V: x(i)+x(j)=0, x(k)=0\ \forall k\in V\setminus \{i,j\}\} && (\text{$e=ij$ is not a loop}) \\
 A_e^d&=\{x\in (\mathbb{R}^d)^V: x(k)=0\ \forall k\in V\setminus \{i\}\} && (\text{$e$ is a loop at $i$}).
 \end{align*}
 Clearly
   \begin{equation}
   \label{eq:A_e1}
   \dim \langle A_e^d: e\in E\rangle=d|V|-\sum_{H\in C(G)} d\lambda(H)\,.
   \end{equation}

 We now establish a variant of equation (\ref{eq:A_e1}).
 Suppose that a $d$-dimensional vector $a_e$ is assigned to each loop $e$.
 We then assign a vector space $B_e$ to $e$ by putting
  \begin{align*}
 B_e&=A_e^d
&& (\text{$e=ij$ is not a loop}) \\
 B_e&=\{x\in (\mathbb{R}^d)^V: x(i)\in \langle a_e\rangle, x(k)=0\ \forall k\in V\setminus \{i\}\} && (\text{$e$ is a loop at $i$})
 \end{align*}
 where $\langle a_e\rangle$ denotes the span of $a_e$.
 Let $\Lo(H)$ be the set of loops in a graph $H$.
 \begin{lemma}
 \label{lem:intercept1}
   \begin{equation}
   \label{eq:A_e}
   \dim \langle B_e: e\in E\rangle=d|V|-\sum_{H\in C(G)} (d-\dim\langle a_e: e\in \Lo(H)\rangle).
   \end{equation}
   \end{lemma}
   \begin{proof}
   This is implicit in \cite{KatTan}, but we give a direct proof since the claim is easy.
   A vector $y\in (\mathbb{R}^d)^V$ is in the orthogonal complement of
   $\langle B_e: e\in E\setminus \Lo(G)\rangle$ if and only if $y(i)=y(j)$ for every $H\in C(G)$ and every $i, j\in V(H)$.
   Such vectors form a $d|C(G)|$-dimensional space.
   Among those vectors, a vector $y$ is in the orthogonal complement of $\langle B_e: e\in E\rangle$ if and only if $y(i)$ is in the orthogonal complement of $\langle a_e: e\in \Lo(H)\rangle$ for every $H\in C(G)$
and every $i\in V(H)$.  Thus the orthogonal complement of $\langle B_e: e\in E\rangle$ has dimension equal to $\sum_{H\in C(G)}(d-\dim\langle a_e: e\in \Lo(H)\rangle)$.
\end{proof}

Another type of subspace which we will associate to each edge is
 \begin{align*}
 C_e&=
 \{x\in \mathbb{R}^V: x(k)=0\ \forall k\in V\setminus \{i,j\}\} && (\text{$e=ij$ is not a loop}) \\
 C_e&=\{x\in \mathbb{R}^V: x(k)=0\ \forall k\in V\setminus \{i\}\} && (\text{$e$ is a loop at $i$}).
 \end{align*}
These subspaces give  a linear representation of the {\em bicircular matroid} of $G$, and we have
\begin{equation}
\label{eq:bicircular}
\dim \langle C_e: e\in E\rangle=|V(E)|.
\end{equation}
We will also need the following result of Lov\'asz \cite{L} which gives a geometric interpretation of the so-called Dilworth truncation of a matroid.
We say that a hyperplane $H$ intersects a family ${\cal U}$ of linear subspaces {\em transversally}
if $\dim H\cap U=\dim U-1$ for every $U\in {\cal U}$.

\begin{lemma}\label{lem:dilworth}
Let $E$ be a finite set and ${\cal U}=\{U_e: e\in E\}$ be a family of linear subspaces of $\mathbb{R}^d$.
Then there exists a hyperplane $H$ which intersects ${\cal U}$ transversally and is such that
\begin{equation}
\label{eq:dilworth}
\dim\langle U_e\cap H\,: e\in E\rangle
= \min \left\{\sum_{i=1}^k (\dim \langle U_e\,:\,e\in E_i\rangle-1)\right\},
\end{equation}
where the minimum is taken over all partitions $\{E_1,E_2,\ldots,E_k\}$ of $E$.
\end{lemma}

\subsubsection{Proof of Theorem~\ref{thm:fixed_intercept0}}
\label{subsec:proof1}
\begin{proof}[Proof of Theorem~\ref{thm:fixed_intercept0}.]
Let $\ell_j=(a_j,0)$ for each $j\in V_L$ and let $R(G,p,\ell)$ be the rigidity matrix of a framework $(G,p,\ell)$ representing the system (\ref{eq:intercept_sys1})--(\ref{eq:intercept_sys2}). This is a $(|V_L|+|E|)\times (2|V_P|+2|V_L|)$-matrix whose rows are of one of the following four types:
\begin{displaymath}\bordermatrix{& & i & & j & &k & & l & \cr
 ij\in E_{PP} & \ldots   & p_i-p_j &  \ldots    & p_j-p_i &  \ldots   & \ldots & \ldots  & \ldots &  \ldots \cr
 jk\in E_{PL} & \ldots   & \ldots &  \ldots    & a_k &  \ldots   & p_j & \ldots  & \ldots &  \ldots \cr
  kl\in E_{LL} & \ldots   & \ldots &  \ldots    & \ldots &  \ldots   & a_l & \ldots  & a_k &  \ldots \cr
  l\in V_{L} & \ldots   & \ldots &  \ldots    & \ldots &  \ldots   & \ldots & \ldots  & a_l &  \ldots \cr
}
\end{displaymath}
 where $i, j$ denote vertices in $V_P$  while $k, l$ denote vertices in $V_L$ (and unspecified entries are equal to zero).
Since the set of the row vectors of $R(G,p,\ell)$ indexed by the vertices in $V_L$ is linearly independent,
$R(G,p,\ell)$ is row-independent if and only if
the projections of the remaining row vectors of $R(G,p,\ell)$ onto the orthogonal complement of the space spanned by the row vectors indexed by $V_L$ form a linearly independent set.
In other words, $R(G,p,\ell)$ is row-independent if and only if the following $|E|\times 2(|V_P|+|V_L|)$-matrix is row-independent:
\begin{displaymath}
\bordermatrix{& & i & & j & &k & & l & \cr
 ij\in E_{PP} & \ldots   & p_i-p_j &  \ldots    & p_j-p_i &  \ldots   & \ldots & \ldots  & \ldots &  \ldots \cr
 jk\in E_{PL} & \ldots   & \ldots &  \ldots    & a_k &  \ldots   & \langle p_j, a_k^{\bot}\rangle a_k^{\bot} & \ldots  & \ldots &  \ldots \cr
  kl\in E_{LL} & \ldots   & \ldots &  \ldots    & \ldots &  \ldots   & \langle a_l, a_k^{\bot}\rangle a_k^{\bot} & \ldots  & \langle a_k, a_l^{\bot}\rangle a_l^{\bot} &  \ldots \cr
 }
 \end{displaymath}

 Note that, if $a_k=a_l$ for $k, l\in V_L$ with $kl\in E_{LL}$, then the corresponding row in the above matrix becomes zero, and hence  $a_k\neq a_l$  is necessary  for $(G,p,\ell)$ to be minimally infinitesimally rigid.
 Thus in the following discussion we assume  $a_k\neq a_l$ for all $kl\in E_{LL}$.

By taking a suitable linear combination of the two columns indexed by each $k\in V_L$ to convert one of these columns to a zero column and then deleting this zero column, and using the fact  that $\langle a_l, a_k^{\bot}\rangle = -\langle a_k, a_l^{\bot}\rangle$ for all pairs $k, l\in V_L$, we may deduce that $R(G,p,\ell)$ is row-independent if and only if  the following $|E|\times (2|V_P|+|V_L|)$-matrix $R'(G,p,\ell)$ is row-independent:
\begin{equation}\label{eq:R'}
\bordermatrix{& & i & & j & &k & & l & \cr
 ij\in E_{PP} & \ldots   & p_i-p_j &  \ldots    & p_j-p_i &  \ldots   & \ldots & \ldots  & \ldots &  \ldots \cr
 jk\in E_{PL} & \ldots   & \ldots &  \ldots    & a_k &  \ldots   & \langle p_j, a_k^{\bot}\rangle & \ldots  & \ldots &  \ldots \cr
  kl\in E_{LL} & \ldots   & \ldots &  \ldots    & \ldots &  \ldots   & 1 & \ldots  & -1 &  \ldots \cr
}
\end{equation}

We will show that there is an injective map $p:V_P\rightarrow \mathbb{R}^2$ such that
$R'(G,p,\ell)$ is row-independent  if and only if
\[|F|\leq 2\nu_{V_P}(F)+\nu_{V_L}(F)-1-
\sum_{H\in C((G[F])^P)}(2-\dim \langle a_j: ij\in F\cap E_{PL}, i\in V(H)\rangle)\]
for all nonempty $F\subseteq E$, implying the theorem.

To this end,  we define the following linear subspace $U_e^P$ in $(\mathbb{R}^2)^{V_P}$ for each $e\in E$:
\begin{align*}
U_e^P&=\{x\in (\mathbb{R}^2)^{V_P}: x(i)+x(j)=0, x(k)=0 \ \forall k\in V_P\setminus \{i,j\} \}
&& (ik\in E_{PP}) \\
U_e^P&=\{x\in (\mathbb{R}^2)^{V_P}: x(i)\in \langle a_j\rangle, x(k)=0 \ \forall k\in V_P\setminus \{i\} \}
&& (ij\in E_{PL}, j\in V_L) \\
U_e^P&=\{0\} && (ij\in E_{LL})
\end{align*}
Note that the linear subspaces are in the form of $B_e$ given in Section~\ref{subsec:matroid} with  the underlying graph $G^P$.
Moreover,  for $H\in C(G^P)$,
there is a correspondence between a loop in $H$ and
an edge $ij\in E_{PL}$ with $i\in V(H)$.
Therefore Lemma~\ref{lem:intercept1} gives
\begin{equation}
\label{eq:intercept1}
\dim \langle U_e^P: e\in E\rangle=2\nu_{V_P}(E)-\sum_{H\in C((G[E])^P)}
(2-\dim \langle a_j: ij\in E_{PL}, i\in V(H)\rangle).
\end{equation}

For each $e\in E$, we also define the following linear subspace $U_e^L$ in $\mathbb{R}^{V_L}$:
\begin{align*}
U_e^L&=\{0\} && (ik\in E_{PP}) \\
U_e^L&=\{x\in \mathbb{R}^{V_L}: x(k)=0 \ \forall k\in V_L\setminus \{j\} \}
&& (ij\in E_{PL}, j\in V_L) \\
U_e^L&=\{x\in \mathbb{R}^{V_L}:  x(k)=0 \ \forall k\in V_L\setminus \{i,j\} \} && (ij\in E_{LL})
\end{align*}
Note that the linear subspaces are in the form of $C_e$ given in Section~\ref{subsec:matroid} with  the underlying graph $G^L$  obtained from $G$ by
removing $V_P$ and regarding each edge $ij$ in $E_{PL}$ with $j\in V_L$ as a loop at $j$.
Hence by (\ref{eq:bicircular})
\begin{equation}
\label{eq:intercept2}
\dim \langle U_e^L: e\in E\rangle=\nu_{V_L}(E).
\end{equation}

Now we consider the direct sum of $(\mathbb{R}^2)^{V_P}$ and $\mathbb{R}^{V_L}$, and let
$U_e$ be the direct sum of $U_e^P$ and $U_e^L$ for each edge $e$.
Combining (\ref{eq:intercept1}) and (\ref{eq:intercept2}),
\begin{equation}
\label{eq:intercept3}
\dim \langle U_e: e\in E\rangle=2\nu_{V_P}(E)+\nu_{V_L}(E)-\sum_{H\in C((G[F])^P)}
(2-\dim \langle a_j: ij\in E_{PL}, i\in V(H)).
\end{equation}
By Lemma~\ref{lem:dilworth}, there is a hyperplane $H$ in $(\mathbb{R}^2)^{V_P}\times \mathbb{R}^{V_L}$ intersecting $\{U_e: e\in E\}$ transversally and satisfying (\ref{eq:dilworth}).
Denote a normal vector of $H$ by $s\in (\mathbb{R}^2)^{V_P}\times \mathbb{R}^{V_L}$.
Since  the hyperplane $H$ intersects $\{U_e: e\in E\}$ transversally,
we may assume $s(i)\neq s(j)$ for $i, j\in V_P$ with $i\neq j$ and $s(j)\neq 0$ for $k\in V_L$
(since a small perturbation of $s$ will not change the property (\ref{eq:dilworth}).)

We define  $p:V_P\rightarrow \mathbb{R}^2$ by $p(i)=s(i)^{\bot}$ and show that $\dim \langle U_e\cap H: e\in E\rangle$ is equal to the rank of $R'(G,p,\ell)$ given in (\ref{eq:R'}).
We will use the following claim, which directly follows from the definition of $U_e$ and the fact that $x\in U_e\cap H$ if and only if $x\in U_e$ and $\langle x,s\rangle=0$.
\begin{claim}
\label{claim:intercept}
 A vector $x\in U_e$ lies in $H$ if and only if:
\begin{itemize}
\item for $e=ij\in E_{PP}$, $x(i)=-x(j)$ and $x(i)$ is proportional to  $p(i)-p(j)$;
\item for $e=ij\in E_{PL}$ wth $j\in V_L$,
$\frac{x(i)}{x(j)}=\frac{s(j)a_j}{\langle p_i, a_j^{\bot}\rangle}$;
\item for $e=ij\in E_{LL}$, $\frac{x(i)}{x(j)}=-\frac{s(j)}{s(i)}$.
\end{itemize}
\end{claim}
Since each $U_e\cap H$ is one-dimensional, Claim~\ref{claim:intercept} implies that
 $\langle U_e\cap H: e\in E\rangle$ is equal to the row space of the $|E|\times (2|V_P|+|V_L|)$-matrix having the following form:
\begin{displaymath}\bordermatrix{& & i & & j & &k & & l & \cr
 ij\in E_{PP} & \ldots   & p_i-p_j &  \ldots    & p_j-p_i &  \ldots   & \ldots & \ldots  & \ldots &  \ldots \cr
 jk\in E_{PL} & \ldots   & \ldots &  \ldots    & a_k &  \ldots   & \langle p_j, a_k^{\bot}\rangle/s(k) & \ldots  & \ldots &  \ldots \cr
  kl\in E_{LL} & \ldots   & \ldots &  \ldots    & \ldots &  \ldots   & 1/s(k) & \ldots  & -1/s(l) &  \ldots \cr
}.
\end{displaymath}
By scaling each column indexed by a vertex in $V_L$, this matrix is transformed to $R'(G,p,\ell)$ (as defined in (\ref{eq:R'})).
In other words,
\begin{equation}
\label{eq:intercept4}
\rank R'(G,p,\ell)=\dim \langle U_e\cap H: e\in E\rangle.
\end{equation}

By (\ref{eq:dilworth}), (\ref{eq:intercept3}), and (\ref{eq:intercept4}), we get
$\rank R'(G,p,\ell) =\min\left\{\sum_{F\in {\cal E}}f(F)\right\}$
where
$$f(F)= 2\nu_{V_P}(F)+\nu_{V_L}(F)-1-\sum_{H\in C((G[F])^P)}
\left(2-\dim \langle a_j: ij\in F\cap E_{PL}, i\in V(H)\rangle\right)$$
and
the minimum is taken over all partitions ${\cal E}$ of $E$ into nonempty subsets.
The function $f:2^E\to \mathbb{Z}$
is submodular, nondecreasing and non-negative, since it determines the dimension of $\langle\{U_e\,:\,e\in F\}\rangle$ by (\ref{eq:intercept3}). Hence $f$ induces
the row matroid of $R'(G,p,\ell)$ by  \cite{E}.
This implies that $\rank R'(G,p,\ell)=|E|$ if and only if $|F|\leq f(F)$
for all nonempty $F\subseteq E$.
\end{proof}

\subsubsection{Proof of Theorem~\ref{thm:fixed_intercept}}
\label{subsec:proof2}
\begin{proof}[Proof of Theorem~\ref{thm:fixed_intercept}.]
By Theorem~\ref{thm:fixed_intercept0} it suffices to prove that the two combinatorial conditions in Theorem~\ref{thm:fixed_intercept} and Theorem~\ref{thm:fixed_intercept0} are equivalent under the assumption that the normals are distinct.

For each edge set $F$ and each $H\in C(G[F]^P)$, recall that  $V(H)$ is a subset of $V(G)$.
We let $F(H)$ be the set of edges in $F$ incident to $V(H)$ in $G$.
Then the counts of Theorem~\ref{thm:fixed_intercept} and Theorem~\ref{thm:fixed_intercept0} can be written as
\begin{align}
\label{eq:int1}
|F|&\leq 2\nu_{V_P}(F)+\nu_{V_L}(F)-1-\max\{0, 2-|V_L(F)|\} \\
|F|&\leq 2\nu_{V_P}(F)+\nu_{V_L}(F)-1-\sum_{H\in C(G[F]^P)} \max\{0, 2-|V_L(F(H))|\},  \label{eq:int2}
\end{align}
respectively, where the count of Theorem~\ref{thm:fixed_intercept0} is simplified to (\ref{eq:int2}) due to the assumption that the normals are distinct.

Since $F(H)\subseteq F$, $\max\{0, 2-|V_L(F(H))|\} \geq \max\{0, 2-|V_L(F)|\}$ for each $H\in C(G[F]^P)$.
Thus (\ref{eq:int2}) implies (\ref{eq:int1}) if $C(G[F]^P)\neq \emptyset$.
If $C(G[F]^P)= \emptyset$, then $F\subseteq E_{LL}$ holds, and hence $|V_L(F)|\geq 2$.
Thus the right hand side of (\ref{eq:int1}) and (\ref{eq:int2}) coincide.
Hence, (\ref{eq:int2}) always implies (\ref{eq:int1}).

To complete the proof, we show that $F$ satisfies (\ref{eq:int2}) if each nonempty  subset of $F$ satisfies (\ref{eq:int1}).
Let $H_1, \dots, H_k$ be all the components in $C(G[F]^P)$ with $|V_L(F(H_i))|\leq 1$.
Let $F'=F\setminus \bigcup_{i=1}^k F(H_i)$.
Then by (\ref{eq:int1}) we have
\begin{align*}
|F(H_i)|&\leq 2\nu_{V_P}(F(H_i))+\nu_{V_L}(F(H_i))-1-\max\{0, 2-|V_L(F(H_i))|\}  \\
|F'|&\leq 2\nu_{V_P}(F')+\nu_{V_L}(F')-1.
\end{align*}
Since $H_i\in C(G[F]^P)$, $V(F(H_i))\cap V(F\setminus F(H_i))\subseteq V_L$ holds,
implying
\[
\nu_{V_P}(F')+\sum_{i=1}^k \nu_{V_P}(F(H_i))=\nu_{V_P}(F).
\]
Moreover,  by $|V_L(F(H_i))|\leq 1$ for $1\leq i\leq k$,  $|V(F(H_i))\cap V(F\setminus F(H_i))|\leq 1$ holds, which implies
\[
\nu_{V_L}(F') +\sum_{i=1}^k \nu_{V_L}(F(H_i)) \leq \nu_{V_L}(F)+k.
\]
Therefore,
\begin{align*}
|F|&=|F'|+\sum_{i=1}^k |F_i| \\
&\leq 2\nu_{V_P}(F')+\nu_{V_L}(F'))-1 +\sum_{i=1}^k (2\nu_{V_P}(F(H_i))+\nu_{V_L}(F(H_i))-1-\max\{0, 2-|V_L(F(H_i))|\} ) \\
&\leq  2\nu_{V_P}(F)+\nu_{V_L}(F)-1+\sum_{i=1}^k \max\{0, 2-|V_L(F(H_i))|\} \\
&= 2\nu_{V_P}(F)+\nu_{V_L}(F)-1+\sum_{H\in C(G[F]^P)} \max\{0, 2-|V_L(F(H_i))|\},
\end{align*}
and $F$ satisfies (\ref{eq:int1}). This completes the proof.
\end{proof}

\subsection{Mixed constraints}\label{subsec:mixed}

A natural question is how to generalise the results of Sections
\ref{subsec:pin}-\ref{subsec:intercept} to the case when the lines
have a mixture of constraints. That is, some lines are completely
fixed, some lines have fixed normals so can translate but not
rotate, some lines can rotate about a fixed point but cannot
translate, and some are unconstrained. We will extend the construction used in the proof of Theorem \ref{thm:fixed_normal} to show that generic instances of this mixed
constraint problem can be transformed to the unconstrained problem
and then solved using Theorem \ref{thm:transfer2}.

Suppose we have a point-line graph $G$ which has various types of
line vertices i.e. a set $V_L^F$ of fixed lines, a set $V_L^N$ of
lines with fixed normals, a collection $\scrr$ of pairwise disjoint
sets of lines with a fixed centre of rotation, and unconstrained
lines. A {\em realisation} of $G$ in $\R^2$ is a framework
$(G,p,\ell)$ together with a map $c:\scrr\to \R^2$, where $c(S)$ is
the centre of rotation for all lines in $S$ for each $S\in \scrr$.
We say that the constrained framework $(G,p,\ell,c)$ is {\em
generic} if the set of coordinates $\{p_i,a_j,c_S\,:\,i\in V_P,j\in
V_L,S\in \scrr\}$ are algebraically independent over $\Q$.

We first consider the case when $|V_L^F|+|\scrr|\geq 1$ and
$|V_L^F|+|\scrr|+|V_L^N|\geq 2$. (In this case no rotation or
translation of $\R^2$ will satisfy the constraints on the lines of
any generic realisation of $G$.) We construct an unconstrained
point-line graph $G'$ by first adding a large rigid point-line graph
$K$ to $G$.
We then choose a line-vertex $v_0$ in $K$ and add an edge from $v_0$ to each
$v\in V_L^N$. This corresponds to the operation of adding the `tree of grey edges' joining the (fixed-normal) line-vertices
in Figure \ref{fig:fixednormal}. 
For each set $S\in \scrr$, we choose a distinct point-vertex $u_S$ in
$K$ and add an edge from $u_S$ to each vertex of $S$. 
This corresponds to the operation of adding a new point-vertex joined by `grey edges' to each of the (fixed-intercept) line-vertices in  Figure \ref{fig:sphereintercept}. Finally
we join each $v\in V_L^F$ to $v_0$ and a point-vertex of $K$.
This construction is illustrated in Figure
\ref{fig:mixed}.

\begin{figure}
\begin{center}
\begin{tikzpicture}[scale=0.8]
\tikzstyle{every node}=[circle, draw=black, very thick, fill=white, inner sep=0pt, minimum width=4pt];




\filldraw[draw=black](0,7.85)-- (0.15,8)-- (0,8.15)-- (-0.15,8)--(0,7.85);
\filldraw[draw=black](0,7.35)-- (0.15,7.5)-- (0,7.65)-- (-0.15,7.5)--(0,7.35);
\filldraw(-0.1,6.4) rectangle (0.1,6.6);

\filldraw[draw=black] (0,5.5) ellipse (0.1cm and 0.15cm);
\filldraw[draw=black] (0,5) ellipse (0.1cm and 0.15cm);
\filldraw[draw=black] (0,4.5) ellipse (0.1cm and 0.15cm);

\filldraw[draw=black] (0,3.5) ellipse (0.15cm and 0.1cm);
\filldraw[draw=black] (0,3) ellipse (0.15cm and 0.1cm);

\node [draw=white, fill=white,rectangle] (a) at (-0.45,8) {$v_1$};
\node [draw=white, fill=white,rectangle] (a) at (-0.45,7.5) {$v_2$};
\node [draw=white, fill=white,rectangle] (a) at (-0.45,6.5) {$v_3$};
\node [draw=white, fill=white,rectangle] (a) at (-0.45,5.5) {$v_4$};
\node [draw=white, fill=white,rectangle] (a) at (-0.45,5) {$v_5$};
\node [draw=white, fill=white,rectangle] (a) at (-0.45,4.5) {$v_6$};
\node [draw=white, fill=white,rectangle] (a) at (-0.45,3.5) {$v_7$};
\node [draw=white, fill=white,rectangle] (a) at (-0.45,3) {$v_8$};

\draw[black,thick]
(2.5,5) -- (4.5,5);

\draw[black,thick]
(4.3,5.2) -- (4.5,5) -- (4.3,4.8);






\filldraw[draw=black] (7,10) circle (3pt);


\filldraw[draw=black] (7,8) circle (3pt);
 \filldraw[draw=black] (7,7.5) circle (3pt);
 \filldraw[draw=black] (7,6.5) circle (3pt);

 \filldraw[draw=black] (7,5.5) circle (3pt);
 \filldraw[draw=black] (7,5) circle (3pt);
 \filldraw[draw=black] (7,4.5) circle (3pt);


 \filldraw[draw=black] (7,3.5) circle (3pt);
 \filldraw[draw=black] (7,3) circle (3pt);


\node [draw=white, fill=white,rectangle] (a) at (6.6,10) {$v_0$};
\node [draw=white, fill=white,rectangle] (a) at (6.6,8) {$v_1$};
\node [draw=white, fill=white,rectangle] (a) at (6.6,7.5) {$v_2$};
\node [draw=white, fill=white,rectangle] (a) at (6.6,6.5) {$v_3$};
\node [draw=white, fill=white,rectangle] (a) at (6.6,5.5) {$v_4$};
\node [draw=white, fill=white,rectangle] (a) at (6.6,5) {$v_5$};
\node [draw=white, fill=white,rectangle] (a) at (6.6,4.5) {$v_6$};
\node [draw=white, fill=white,rectangle] (a) at (6.6,3.5) {$v_7$};
\node [draw=white, fill=white,rectangle] (a) at (6.6,3) {$v_8$};



\node (u1) at (11,7) {};
\node (u2) at (11,5) {};
\node [draw=white, fill=white,rectangle] (a) at (11.5,7) {$u_{S_1}$};
\node [draw=white, fill=white,rectangle] (a) at (11.5,5) {$u_{S_2}$};

\draw[black,thick]
(7,10) -- (u1) -- (u2) -- (7,10) -- (7,8);

\draw[thick] plot[smooth, tension=1] coordinates{(7,10)(7.5,8.75)(7,7.5)};

\draw[thick] plot[smooth, tension=1] coordinates{(7,10)(7.7,8.25)(7,6.5)};

\draw[black,thick]
(7,6.5) --(u1) -- (7,5.5);

\draw[black,thick]
(7,4.5) -- (u1) -- (7,5);

\draw[black,thick]
(7,3) -- (u2) -- (7,3.5);

\draw (0,5.5) ellipse (1.2cm and 3.6cm);

\draw (7,5.5) ellipse (1.2cm and 3.6cm);

\end{tikzpicture}
\caption{A constrained point-line graph $G$ with eight constrained line vertices: $v_1$ and $v_2$ have fixed normals; $v_3$ is fixed; $S_1=\{v_4,v_5,v_6\}$ and $S_2=\{v_7,v_8\}$ have fixed centres of rotation. We transform $G$ to an unconstrained point-line graph $G'$ by adding the rigid graph $K$ with two point-vertices, $u_{S_1}$ and $u_{S_2}$, and one line-vertex $v_0$.}
\label{fig:mixed}
\end{center}
\end{figure}
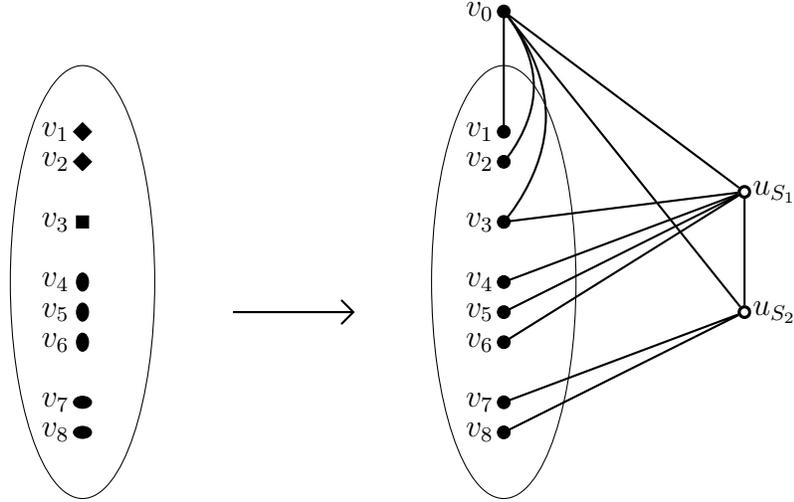


Let $(G, p,\ell,c)$ and $(G', p',\ell')$ be  generic realisations of
$G$ and $G'$ in $\R^2$ so that $p(u)=p'(u)$ for all $u\in V_P$,
$\ell(v)=\ell'(v)$ for all $v\in V_L$,  and $c(S)=p'(u_S)$ for all
$S\in \scrr$. Then $(G, p,\ell,c)$ has a non-zero infinitesimal
motion if and only if $(G', p',\ell')$ has a non-zero infinitesimal
motion which keeps $K$ fixed. Hence $(G, p,\ell,c)$ is
infinitesimally rigid as a constrained point-line framework if and
only if $(G', p',\ell')$ is infinitesimally  rigid as an
unconstrained point-line framework. Thus we may determine whether
$(G, p,\ell,c)$ is infinitesimally rigid by applying Theorem
\ref{thm:transfer2} to $(G', p',\ell')$. Note that our definition of the genericity of
$(G, p,\ell,c)$ is independent of the choice of $r_j$ for $j\in
V_L$. (This makes sense because the structure of the rigidity matrix given in
Section \ref{subsec:rigmat} implies that the rank of $R_\A(G,p',\ell')$ will
be maximised for any realisation such that the coordinates of $p_i,a_j,c_S$ are algebraically independent.) This means that we are free to
choose the values of the $r_j$ such that, for each $S\in \scrr$,
each of the lines $\ell(v)$ with $v\in S$ passes through the the
point $c(S)$, so we can take the lines in each $S$ to be concurrent
as in Section \ref{subsec:intercept} if we wish.

Similar, but simpler constructions, can be used when $|V_L^F|+|\scrr|+|V_L^N|\leq 1$ or $|V_L^F|+|\scrr|= 0$.

\section{Further remarks and open problems}

The combinatorial condition in Theorem \ref{thm:barjointcol}(c), Theorem \ref{thm:fixed_normal}(d),(e) and Theorem \ref{thm:fixedline} can all be checked in polynomial time,
see \cite{JO}, \cite{Im,JO,Su} and \cite{ShaiServWh}, respectively.
For the condition in Theorem~\ref{thm:fixed_intercept},
as shown in the proof, the right hand side of the count condition defines a submodular function $f$, and hence one can decide whether the condition is satisfied in polynomial time by a general submodular function minimization algorithm. Currently we do not have a specialized efficient algorithm for this count.

We have obtained characterizations of fixed-line rigidity and fixed-intercept rigidity in Theorem~\ref{thm:fixedline} and Theorem~\ref{thm:fixed_intercept0} for a point-line framework with {\em arbitrary} normals for its lines.
An interesting  open problem is to derive an analogous result for fixed-normal rigidity.
An important special case is the problem of characterizing fixed-normal rigidity for point-line frameworks in which the lines have been partitioned into parallel classes with generic normals (this was posed by Bill Jackson and John Owen at the rigidity workshop in Banff in 2015).
In view of the relationship between fixed normal rigidity and scene analysis described in Section 3,
this problem is challenging even when the underlying graph
is naturally bipartite (as it is equivalent to understanding when an arbitrary $2$-scene has only trivial realisations).
We have constructed examples of (nongeneric) 2-dimensional naturally bipartite point-line frameworks with distinct line-normals
which satisfy the count condition of Theorem~\ref{thm:fixed_normal} but  are not fixed-normal rigid.

In \cite{Izmestiev,BSWWSphere} the transfer of rigidity results between Euclidean spaces and spherical spaces was extended to Minkowski spaces and hyperbolic spaces (spheres in Minkowski space).  Preliminary versions of these infinitesimal rigidity transfers appear in \cite{SaliolaWh}.  Applying these transfers,  Theorem 2.5 transfers to frameworks with a designated coplanar subframework in Minkowski space and in hyperbolic space. However there are a number of further interesting questions in this direction such as extending our other results on point-hyperplane frameworks with various constraints
on the hyperplanes.  Since Minkowski space has the full space of translations for hyperplanes, it is a natural setting to extend the results
of this paper.

In the setting of body-bar frameworks, including the specific setting of Body CAD frameworks, there are preliminary results  \cite{Lee, KatTan} which include (nongeneric) pinned and slider-joints, and point-line, as well as point-plane, distance constraints.   These results and the results in this paper, can be refined and extended to explore body-bar-point frameworks - which do occur in built linkages in $2$D and $3$D.  Some of these extensions and further related results will appear in the thesis \cite{Y}.

Previous work on direction-length frameworks \cite{ServWhiteley, WhDirection} can be viewed as (i) combinatorially special fixed-normal point-line frameworks, with exactly two points joined to each line; and (ii) all point-line distances set as $0$ i.e. they are point-line  {\em incidences}.   With all this special geometry, the combinatorial characterization in \cite{ServWhiteley} has one more condition than the characterization for general point-line frameworks in Theorem~\ref{thm:fixednormal}.  This added condition comes from the fact that subgraphs with no point-point distance constraints
can be dilated. An extension in which fixed normal lines are allowed to contain an arbitrary number of points is given in \cite{OP}. Many other extensions are open for exploration.  These connections also suggest that prior results on parallel-drawings and mixed frameworks in higher dimensions, again with two vertices per line, such as \cite{CrWhParallel, WhParallel}, can be generalized to combinatorially special  fixed-normal point-hyperplane frameworks in higher dimensions.

\section*{Acknowledgements}

We thank Banff International Research Station for hosting the 2015 workshop on `Advances in Combinatorial and Geometric Rigidity' during which this work was started. We also thank HIM (Bonn) and ICMS (Edinburgh) for hosting further rigidity workshops in October 2015 and May 2016, respectively, and we thank DIMACS (Rutgers) for hosting the workshop on `Distance Geometry' in July 2016. The discussions we had during these workshops played a crucial role in the writing of this paper.


\end{document}